\documentclass[11pt]{article}
\usepackage{latexsym,amsmath,amssymb,paralist,subfigure,graphicx,array}
\usepackage{multirow}
\usepackage{comment}
\usepackage[compress]{cite}
\usepackage{algorithm}
\usepackage{algpseudocode}
\usepackage[title]{appendix}
\usepackage{xcolor}
\usepackage{BOONDOX-calo} 

\numberwithin{equation}{section}

\usepackage{graphicx}
\usepackage{epstopdf}
\usepackage{bm}
    \usepackage{nicefrac}
    \usepackage{longtable}
\usepackage{hyperref}
\usepackage[normalem]{ulem}

\newcommand{\reff}[1]{{\rm (\ref{#1})}}

\newcommand{\R}{\mathbb{R}}            

\newcommand{\K}{\mathbb{K}}
\newcommand{\I}{\mathbb{I}}
\newcommand{\N}{\mathbb{N}}
\newcommand{\ve}{\varepsilon}          
\newcommand{\bn}{\mathbf n}            

\newcommand{\nabh}{\nabla_{\! h}}

\newcommand{\calG}{\mathcal G}

\newcommand{\cJ}{\mathcal J}
\newcommand{\calC}{\mathcal C}
\newcommand{\ckC}{\check C}

\newcommand{\tdC}{\td C}
\newcommand{\calD}{\mathcal D}
\newcommand{\calA}{\mathcal A}

\newcommand{\calL}{\mathcal L}
\newcommand{\calV}{\mathcal V}

\newcommand{\cphi}{\check \phi}
\newcommand{\crho}{\check \rho}

\newcommand{\x}{\mbox{\boldmath$x$}}

\newcommand{\hf}{\nicefrac{1}{2}}
\newcommand{\nrm}[1]{\left\| #1 \right\|}
\newcommand{\ciptwo}[2]{\left\langle #1 , #2 \right\rangle}

\newcommand\dt {{\Delta t}}

\newcommand{\eipx}[2]{\left[ #1 , #2 \right]_{\rm x}}
\newcommand{\eipy}[2]{\left[ #1 , #2 \right]_{\rm y}}
\newcommand{\eipz}[2]{\left[ #1 , #2 \right]_{\rm z}}
\newcommand{\eipvec}[2]{\left[ #1 , #2 \right]}

\newtheorem{theorem}{Theorem}[section]
\newtheorem{lemma}[theorem]{Lemma}
\newtheorem{proposition}[theorem]{Proposition}

\newtheorem{remark}[theorem]{Remark}

\newenvironment{proof}[1][Proof]{\begin{trivlist}
\item[\hskip \labelsep {\bfseries #1}]}{\end{trivlist}}

\newcommand{\qed}{\nobreak \ifvmode \relax \else
      \ifdim\lastskip<1.5em \hskip-\lastskip
      \hskip1.5em plus0em minus0.5em \fi \nobreak
      \vrule height0.75em width0.5em depth0.25em\fi}

\def\XXint#1#2#3{{\setbox0=\hbox{$#1{#2#3}{\int}$}
\vcenter{\hbox{$#2#3$}}\kern-.51\wd0}}

\newcommand{\td}{\tilde}

{\allowdisplaybreaks

\begin{document}

\graphicspath{{Figures/}}

\title{A second-order accurate, original energy dissipative numerical scheme for chemotaxis and its convergence analysis}

\author{Jie Ding\thanks{
 School of Science, Jiangnan University,  Wuxi, Jiangsu, 214122, China. E-mail: jding@jiangnan.edu.cn.}
\and
Cheng Wang\thanks{Department of Mathematics, University of Massachusetts, North Dartmouth, MA  02747, USA. E-mail: cwang1@umassd.edu.}
\and	
Shenggao Zhou\thanks{School of Mathematical Sciences, MOE-LSC, and CMA-Shanghai, Shanghai Jiao Tong University, Shanghai, 200240, China. E-mail: sgzhou@sjtu.edu.cn.}
}
\maketitle

\begin{abstract}
This paper proposes a second-order accurate numerical scheme for the Patlak--Keller--Segel system with various mobilities for the description of chemotaxis. Formulated in a variational structure, the entropy part is novelly discretized by a modified Crank-Nicolson approach so that the solution to the proposed nonlinear scheme corresponds to a minimizer of a convex functional.  A careful theoretical analysis reveals that the unique solvability and positivity-preserving property could be theoretically justified. More importantly, such a second order numerical scheme is able to preserve the dissipative property of the original energy functional, instead of a modified one. To the best of our knowledge, the proposed scheme is the first second-order accurate one in literature that could achieve both the numerical positivity and original energy dissipation. In addition, an optimal rate convergence estimate is provided for the proposed scheme, in which rough and refined error estimate techniques have to be included to accomplish such an analysis. Ample numerical results are presented to demonstrate robust performance of the proposed scheme in preserving positivity and original energy dissipation in blowup simulations.

\vspace{4mm}



\end{abstract}


\section{Introduction}
As a classical chemotaxis model, the Patlak--Keller--Segel (PKS) system is often used to describe the evolution of living organisms interacting with environmental signals~\cite{Patlak_BMB1953, KellerSegel_JTB1971}:
\begin{equation}\label{PKSeqns}
\left\{
\begin{aligned}
\partial_t \rho&=\gamma\Delta \rho-\chi\nabla\cdot(\eta(\rho)\nabla\phi),\\
\theta\partial_t\phi&=\mu\Delta\phi-\alpha\phi+\chi\rho.
\end{aligned}
\right.
\end{equation}
Here $\rho$ is the density distribution of living organisms, $\phi$ stands for the density of the chemical signals, $\gamma$, $\mu$, and $\alpha$ are three positive constants, $\chi$ denotes the chemotactic sensitivity, $\eta(\rho)$ is the density-dependent mobility, and $\theta\geq 0$ describes how fast chemical signals response to living organisms. To address the confinement effect in a bounded domain $\Omega$, the PKS system is prescribed with homogeneous Neumann boundary condition:
\begin{equation} \label{BCs}
\nabla\rho\cdot\bn=\nabla\phi\cdot\bn=0,~~~\mbox{on}~~\partial\Omega.
\end{equation}

The PKS system describes the diffusion of living organisms and aggregation induced by chemical signals.  In particular, the nonlinear term $-\chi\nabla\cdot(\eta(\rho)\nabla\phi)$ models organism movement towards higher density of chemical signals.
It is well-known that the solution of the classical PKS system~\reff{PKSeqns} with $\eta(\rho)=\rho$ may blow up in finite time. Many efforts have been devoted to mathematical analysis on blow-up solutions~\cite{Nagai_JIA2001, Horstmann_DMV2004, Horstmann_DMV2003, Perthame_BB2006, NagaiSenbaYoshida_FE1997, BlanchetCarlenCarrillo_JFA2012}.  According to the homogeneous Neumann boundary condition~\reff{BCs}, the total mass is conserved in the PKS system. There is a certain critical threshold value for initial total mass, by which the finite-time blow up solution and globally existent solution can be distinguished~\cite{NagaiSenbaYoshida_FE1997, BlanchetCarlenCarrillo_JFA2012, BlanchetDolbeaultPerthame_EJDE2006, LiuWang_AML2016, CalvezCorrias_CMS2008, JgerLuckhaus_TAMS1992}. However, the density of living organisms does not blow up in reality, rather exhibiting density peaks with difference of several orders in magnitude. Modified models with various $\eta(\rho)$ have been proposed to capture such a phenomenon \cite{Kurganov_DCDSS2014}. 

Many numerical methods have been proposed for the PKS system in various chemotaxis applications \cite{AcostaGonzalezGalvan_JSC2023, BadiaBonillaSantacreu_MMMAS2023, WangZhouShiChen_JCP2022, Epshteyn_JSC2012, ShenXu_2020SINA, LiuWangZhou_2016MC, Filbet_NM2006, ZhouSaito_NM2017, ChatardJungel_IMA2012}. The solution to the PKS system has several properties of great physical significance, such as mass conservation, positivity for cell density, and energy dissipation. Shen and Xu develop unconditionally energy-stable schemes that preserve positivity/bounds for the PKS equations~\cite{ShenXu_2020SINA}. Based on the scalar auxiliary variable (SAV) approach,  a high-order, linear, positivity/bound preserving and unconditionally energy-stable scheme has also been developed in~\cite{HuangShen_2021SISC}. Based on the Slotboom formulation,  a positivity-preserving and asymptotic preserving scheme has also been constructed for the PKS system in 2D~\cite{LiuWangZhou_2016MC}.  
On the other hand, second-order positivity-preserving central-upwind schemes have been developed for chemotaxis models, by using the finite volume method~\cite{ChertockKurganov_NM2008} and discontinuous Galerkin approach~\cite{EpshteynKurganov_SINM2008, EpshteynIzmirlioglu_JSC2009}. 
An implicit finite volume scheme has been proposed in~\cite{Filbet_NM2006}, in which the existence of a positive solution is established under certain restrictions. Bessemoulin-Chatard and J${\rm \ddot{u}}$ngel have also constructed a finite volume scheme for the PKS model \cite{ChatardJungel_IMA2012}, with an additional cross-diffusion term in the second equation of \reff{PKSeqns}. The positivity-preserving, mass conservation, entropy stability, and the well-posedness of the nonlinear scheme have been proved in the work. Zhou and Saito have introduced a linear finite volume scheme that satisfies both positivity and mass conservation requirements~\cite{ZhouSaito_NM2017}.

Because of the non-constant mobilities, the numerical design of a second order accurate in time, energy-stable algorithm for the PKS system turns out to be very challenging. In this work, we propose a novel second-order (in time) numerical scheme for the PKS equations. The standard Crank-Nicolson approximation is applied to the chemoattractant evolution equation, while the variational structure of the density equation is used to facilitate the numerical design. In more details, the mobility function is computed by an explicit second order extrapolation formula, and such an explicit treatment will be useful in the unique solvability analysis. On the other hand, a singular logarithmic term appears in the chemical potential, and poses a great challenge for second-order temporal discretization to ensure the theoretical properties at a discrete level. To overcome this subtle difficulty, we approximate the logarithmic term by a careful Taylor expansion, up to the second order accuracy. The unique solvability and positivity-preserving property of such a highly nonlinear and singular numerical system is theoretically established, in which the convex and singular nature of the implicit terms play a very important role in the theoretical analysis; see the related works for various gradient flow models with singular energy potential~\cite{ChenJingWangWang_JSC22, ChenJingWangWangWise_CiCP22, ChenWangWangWise_JCP2019, DongWangWiseZhang_JCP21, DongWangZhangZhang_CMS19, DongWangZhangZhang_CiCP20, YuanChenWangWiseZhang_JSC21, ZhangWangWiseZhang_SISC20}. This approach also avoids a nonlinear artificial regularization term in the numerical design. More importantly, a careful nonlinear analysis reveals a dissipation property of the original free energy functional, instead of a modified energy energy reported in many existing works for multi-step numerical schemes~\cite{ShenXu_2020SINA,HuangShen_2021SISC}. This turns out to be a remarkable theoretical result for a second order accurate scheme.  


It is observed that, a highly nonlinear formulation in the numerical system, which is designed to accomplish certain structure-preserving properties at a discrete level, often  poses a challenging task for a rigorous convergence analysis. In this work, an optimal rate convergence analysis is performed for the proposed second-order numerical scheme. Due to the non-constant mobility nature, together with the highly nonlinear and singular properties of the logarithmic terms, such an optimal rate convergence analysis for the PKS equations turns out to be a very complicated issue. To overcome this difficulty, several highly non-standard techniques have to be introduced. A careful linearization expansion is required for the higher-order asymptotic analysis of the numerical solution, up to the fourth order accuracy in both time and space. Such a higher-order asymptotic expansion enables one to derive a maximum norm bound for the density variable, based on a rough error estimate. Subsequently, the corresponding inner product between the discrete temporal derivative of the numerical error function and the numerical error associated with the logarithmic terms becomes a discrete derivative of certain nonlinear, non-negative functional in terms of the numerical error functions, along with a few numerical perturbation terms. Consequently, all the major challenges in the nonlinear analysis of the second-order accurate scheme could be overcome, and the associated error estimate could be carefully derived. 
To our knowledge, this is the first work to combine three theoretical properties for second-order accurate numerical schemes for the PKS system: positivity-preservation, original energy dissipation, and optimal rate convergence analysis.

The rest of the paper is organized as follows. In Section 2,  the PKS system for chemotaxis is introduced, and the associated physical properties are recalled. Subsequently, a second-order accurate numerical scheme is proposed in Section 3. Afterward, the structure-preserving properties of the proposed numerical scheme, such as mass conservation, unique solvability, positivity-preserving property, and the original energy dissipation, are proved in Section 4. In addition, the optimal rate convergence analysis is presented in Section 5. Some numerical results are provided in Section 6. Finally, some concluding remarks are given in Section 7.

\section{Chemotaxis models}
For the PKS system~\eqref{PKSeqns}, the following free energy is considered:
\begin{equation} 
F(\rho,\phi)=\int_{\Omega}\left[\gamma f(\rho)-\chi\rho\phi+\frac{\mu}{2}|\nabla\phi|^2+\frac{\alpha}{2}\phi^2\right]d\x . \label{energy-0} 
\end{equation} 
For simplicity of presentation, the notation $\eta(\rho)=\frac{1}{f''(\rho)}$ is introduced. With an alternate representation formula $\Delta\rho=\nabla\cdot(\frac{1}{f''(\rho)}\nabla f'(\rho))$~\cite{ShenXu_2020SINA}, the PKS system could be rewritten as
\begin{equation}\label{KSeqn}
\left\{
\begin{aligned}
\partial_t \rho&=\nabla\cdot\left(\frac{1}{f''(\rho)}\nabla\left[\gamma f'(\rho)-\chi\phi\right]\right)=\nabla\cdot\left(\frac{1}{f''(\rho)}\nabla\frac{\delta F}{\delta \rho}\right),\\
\theta\partial_t\phi&=\mu\Delta\phi-\alpha\phi+\chi\rho=-\frac{\delta F}{\delta \phi}.
\end{aligned}
\right.
\end{equation}
The homogeneous Neumann boundary condition~\eqref{BCs} is imposed for both physical variables.

In this work, we consider three typical choices of the entropy function $f(\rho)$ and the corresponding mobility function $\eta(\rho)$ (see the related examples in \cite{ShenXu_2020SINA, HillenPainter_JMB2009, Perthame_BB2006}): 
\begin{itemize}
\item (i) The classical PKS system:
\begin{equation}\label{f_1}
\eta(\rho)=\rho,~~ f(\rho)=\rho(\ln \rho-1) , \quad  \mbox{for $\rho\in \I= (0,\infty)$} ; 
\end{equation}
\item (ii) PKS system with a bounded mobility \cite{Velazquez_SIAM2004_1, Velazquez_SIAM2004_2}:
\begin{equation}\label{f_2}
\eta(\rho)=\frac{\rho}{\kappa\rho+1}(\kappa>0), ~~f(\rho)=\rho(\ln \rho-1)+\kappa \rho^2/2 , \quad  \mbox{for $\rho\in \I= (0,\infty)$} ;
\end{equation}
\item (iii) PKS system with a saturation density \cite{HillenPainter_2001, DolakSchmeiser_SIAM2005}:
\begin{equation}\label{f_3}
\eta(\rho)=\rho(1-\rho/M)(M>0), ~~f(\rho)=\rho \ln \rho+(M-\rho) \ln (1-\rho/M) , \quad 
 \mbox{for $\rho\in \I= (0, M)$} , 
\end{equation}
where $M$ is the saturation density.
\end{itemize}
For cases (i) and (ii), the solution of the PKS equations requires the positivity of the density variable, while in the case (iii),  a bound $0 < \rho < M$ is needed. 

With homogeneous Neumann boundary condition, the analytical solution to \reff{PKSeqns} has three properties of physical importance:
\begin{itemize}
\item  Mass conservation: the total density remains constant over time, i.e.,\\
\[
\int_{\Omega} \rho(\x,t)d\x=\int_{\Omega} \rho(\x,0)d\x , \quad \forall t > 0 ;
\]
\item Bound/Positivity: the organism density is positive, i.e.,\\
\[
\rho(\x,t)\in \I,~~\mbox{if}~~ \rho(\x,0)\in \I,~~\mbox{for}~\x\in\Omega,~\forall t>0;
\]
\item Free-energy dissipation: the free-energy~\reff{energy-0} decays in time \cite{BlanchetCarrilloLisini_M2AN2014, CongLiu_DCDSB2016}\\
\[
\begin{aligned}
\frac{dF}{dt}=-\int_{\Omega}\left[\frac{1}{f''(\rho)}\left(\nabla\frac{\delta F}{\delta\rho}\right)^2+\frac{1}{\theta}\left(\frac{\delta F}{\delta\phi}\right)^2 \right]d\x \leq 0 ,~~\mbox{for}~ \theta >0. 
\end{aligned}
\]

\end{itemize}

\section{The numerical scheme}

\subsection{Some notations}
For simplicity, a cubic rectangular computational domain $\Omega= (a, b)^3$ is considered, with homogeneous Neumann boundary condition. Let $N\in \mathbb{N}^*$ be the number of grid points along each dimension, and $h=(b-a)/N$ be the uniform grid spacing size. The computational domain is covered by the cell-centered grid points
\[
\left\{x_i, y_j, z_k\right\}= \left\{a+(i-\frac{1}{2})h, a+(j-\frac{1}{2})h, a+(k-\frac{1}{2})h\right\} , 
\]
for $i,j,k= 1, \cdots, N$.
Denote by $\rho_{i,j,k}$ and $\phi_{i,j,k}$ the discrete approximations of $\rho(x_i,y_j,z_k, \cdot)$ and $\phi(x_i,y_j,z_k, \cdot)$, respectively.

The standard discrete operators and notations are recalled in the finite difference discretization~\cite{wise09a, wang11a}. The following grid function spaces, with homogeneous Neumann boundary condition, are introduced:
\begin{equation}
\begin{aligned}
\mathcal{C}:=&\big\{u\big|u_{m N,j,k}=u_{1+m N,j,k},~u_{i,m N,k}=u_{i,1+m N,k},~u_{i,j,m N}=u_{i,j,1+m N}~~ \\
&\qquad \forall i,j,k=1,2,\cdots,N, m=0,1 \big\},\\
\end{aligned}
\end{equation}
\[
\mathring{\mathcal{C}}:=\bigg\{u\in \mathcal{C}\bigg| \overline{u}=0 \bigg\}, \quad
 \overline{u}=\frac{h^3}{|\Omega|}\sum_{i,j,k=1}^N u_{i,j,k} .
\]
Meanwhile, the average and difference operators in the $x$-direction are given by 
	\begin{eqnarray*}
&& A_x f_{i+\hf,j,k} := \frac{1}{2}\left(f_{i+1,j,k} + f_{i,j,k} \right), \quad D_x f_{i+\hf,j,k} := \frac{1}{h}\left(f_{i+1,j,k} - f_{i,j,k} \right), \\
&& a_x f_{i, j, k} := \frac{1}{2}\left(f_{i+\hf, j, k} + f_{i-\hf, j, k} \right),	 \quad d_x f_{i, j, k} := \frac{1}{h}\left(f_{i+\hf, j, k} - f_{i-\hf, j, k} \right).
	\end{eqnarray*}
Average and difference operators in $y$ and $z$ directions, denoted by $A_y$, $A_z$, $D_y$, $D_z$, $a_y$, $a_z$, $d_y$, and $d_z$, could be analogously defined. The discrete gradient and discrete divergence become 
	\[
	\begin{aligned}
&\nabh f_{i,j,k} =\left( D_xf_{i+\hf, j, k},  D_yf_{i, j+\hf, k},D_zf_{i, j, k+\hf}\right) , \\
&\nabh\cdot\vec{f}_{i,j,k} = d_x f^x_{i,j,k}	+ d_y f^y_{i,j,k} + d_z f^z_{i,j,k},
\end{aligned}
	\]
where $\vec{f} = (f^x,f^y,f^z)$, with $f^x$, $f^y$ and $f^z$ evaluated at $(i+1/2,j,k), (i,j+1/2,k), (i,j,k+1/2)$, respectively. The standard discrete Laplacian turns out to be 
	\begin{align*}
\Delta_h f_{i,j,k} := & \nabla_{h}\cdot\left(\nabla_{h}f\right)_{i,j,k} =  d_x(D_x f)_{i,j,k} + d_y(D_y f)_{i,j,k}+d_z(D_z f)_{i,j,k}.
	\end{align*}
Similarly, for a scalar function $\mathcal{D}$ that is defined at face center points, 
we have
	\[
\nabla_h\cdot \big(\mathcal{D} \vec{f} \big)_{i,j,k} = d_x\left(\mathcal{D}f^x\right)_{i,j,k}  + d_y\left(\mathcal{D}f^y\right)_{i,j,k} + d_z\left(\mathcal{D}f^z\right)_{i,j,k} .
	\]
If $f\in \mathcal{C}$, then $\nabla_h \cdot\left(\mathcal{D} \nabla_h  \cdot \right):\mathcal{C} \rightarrow \mathcal{C}$ becomes
	\[
\nabla_h\cdot \big(\mathcal{D} \nabla_h f \big)_{i,j,k} = d_x\left(\mathcal{D}D_xf\right)_{i,j,k}  + d_y\left(\mathcal{D} D_yf\right)_{i,j,k} + d_z\left(\mathcal{D}D_zf\right)_{i,j,k} .
	\]
	
For $f,g\in\calC$, the discrete $L^2$ inner product is defined as
\[
\ciptwo{f}{\xi}  := h^3\sum_{i,j,k=1}^N  f_{i,j,k}\, \xi_{i,j,k},\quad f,\, \xi\in {\mathcal C}.
\]	
Similarly, for two vector grid functions $\vec{f}_i = (f_i^x,f_i^y,f_i^z)$, $i=1,2$, evaluated at $(i+1/2,j,k), (i,j+1/2,k), (i,j,k+1/2)$, respectively, the corresponding discrete inner product is given by
	\begin{equation*}
	\begin{aligned}
&\eipx{f}{\xi} := \langle a_x(f\xi) , 1 \rangle ,\quad
\eipy{f}{\xi} := \langle a_y(f\xi) , 1 \rangle ,\quad \\
&\eipz{f}{\xi} := \langle a_z(f\xi) , 1 \rangle ,\quad
[ \vec{f}_1 , \vec{f}_2 ] : = \eipx{f_1^x}{f_2^x}	+ \eipy{f_1^y}{f_2^y} + \eipz{f_1^z}{f_2^z}.
	\end{aligned}
	\end{equation*}	
In turn, the following norms could be introduced for $f\in {\mathcal C}$: $\nrm{f}_2^2 := \langle f , f \rangle$, $\nrm{f}_p^p := \ciptwo{|f|^p}{1}$, with $1\le p< \infty$, and $\nrm{f}_\infty := \max_{1\le i,j,k\le N}\left|f_{i,j,k}\right|$. The gradient norms are defined as
	\begin{eqnarray*}
\nrm{ \nabla_h f}_2^2 &: =& \eipvec{\nabh f }{ \nabh f } = \eipx{D_xf}{D_xf} + \eipy{D_yf}{D_yf} +\eipz{D_zf}{D_zf},  \quad \forall \,  f \in{\mathcal C} ,
	\\
\nrm{\nabla_h f}_p^p &:=&  \eipx{|D_xf|^p}{1} + \eipy{|D_yf|^p}{1} +\eipz{|D_zf|^p}{1}   , \quad \forall \, f \in{\mathcal C}, \quad  1\le p<\infty .
	\end{eqnarray*}
The higher-order norms could be similarly introduced:
	\[
\nrm{f}_{H_h^1}^2 : =  \nrm{f}_2^2+ \nrm{ \nabla_h f}_2^2, \quad \nrm{f}_{H_h^2}^2 : =  \nrm{f}_{H_h^1}^2  + \nrm{ \Delta_h f}_2^2  , \quad \forall \, f \in{\mathcal C}.
	\]
	
We now define a discrete analogue of the space $H^{-1}(\Omega)$. Consider a positive, scalar function $\calD$. For any $g\in\mathring{\calC}$, there exists a unique solution $f\in\mathring{\calC}$ to the equation
\[
\calL_{\calD} f:=-\nabla_h\cdot(\calD\nabla_h f)=g,
\]	
with discrete homogeneous Neumann boundary condition 
\[
f_{m N,j,k}=f_{1+m N,j,k},~f_{i,m N,k}=f_{i,1+m N,k},~f_{i,j,m N}=f_{i,j,1+m N}~~\mbox{for}~i,j,k=1,\cdots,N,~m=0,1.\\
\]
Then the following discrete norm could be introduced:
\[
\|g \|_{\calL^{-1}_{\calD}}=\sqrt{\ciptwo{g}{\calL^{-1}_{\calD}(g) }}.
\]
In particular, if $\calD=1$, we have $\calL_{1} f=-\Delta_h f$, and a discrete $\| \cdot \|_{-1,h}$ norm becomes
\[
\|g \|_{-1,h}=\sqrt{\ciptwo{g}{(-\Delta_h)^{-1}(g)}}.
\]

\begin{lemma}\cite{wang11a, wise09a, guo16}
For any $\phi_1$, $\phi_2$, $\phi_3$, $g\in\calC$, and any $\vec{f} = (f^x,f^y,f^z)$, with $f^x$, $f^y$ and $f^z$ evaluated at $(i+1/2,j,k), (i,j+1/2,k), (i,j,k+1/2)$, respectively, the following summation-by-parts formulas are valid:
\[
\ciptwo{\phi_1}{\nabla_h\cdot\vec{f}}=-[\nabla_h\phi_1,\vec{f}],~~\ciptwo{\phi_2}{\nabla_h\cdot(g\nabla_h\phi_3)}=-[\nabla_h\phi_2,\calA_h(g)\nabla_h \phi_3],
\]
in which $\calA_h$ corresponds to the average operator given by $A_x$, $A_y$, and $A_z$.
\end{lemma}


\subsection{Second-order accurate numerical scheme}
 A uniform time step size $\Delta t$ is taken, so that $t_n=t_{n-1}+\Delta t$. For $\rho^n$, $\phi^n\in\calC$, the PKS system~\reff{PKSeqns} is discretized by
\begin{align}
\frac{\rho^{n+1}-\rho^n}{\Delta t}=&\nabla_h\cdot \Big[ \frac{1}{f''(\hat{\rho}^{n+\frac{1}{2}})}\nabla_h \Big(\gamma S^{n+\frac{1}{2}}-\frac{\chi}{2}(\phi^{n+1}+\phi^n) 
+ \frac{\chi^2 \dt}{4\theta}  (\rho^{n+1}-\rho^n) \Big) \Big],\label{PKS2nd_1}\\
\theta\frac{\phi^{n+1}-\phi^n}{\Delta t}=&\frac{\mu}{2}\Delta_h(\phi^n+\phi^{n+1})-\frac{\alpha}{2}(\phi^{n+1}+\phi^n)+\frac{\chi}{2}(\rho^{n+1}+\rho^n),\label{PKS2nd_2}
\end{align}
where
\begin{equation}\label{Nolite}
S^{n+\frac{1}{2}}=f'(\rho^{n+1})-\frac{1}{2}f''(\rho^{n+1})(\rho^{n+1}-\rho^n)+\frac{1}{6}f'''(\rho^{n+1})(\rho^{n+1}-\rho^n)^2 .
\end{equation}
In particular, the mobility function at $t_{n+\frac12}$, namely $\frac{1}{f''(\hat{\rho}^{n+\frac{1}{2}})}$, is approximated by  
\begin{equation}\label{MidAp}
\hat{\rho}^{n+\frac{1}{2}}=\Big((\frac{3}{2}\rho^{n}-\frac12\rho^{n-1})^2+\dt^8 \Big)^{\frac12} ,
\end{equation}
to ensure both the positivity and a higher order consistency. See related derivation in~\cite{LiuWang_JSC2023}. 

Define two linear, invertible operators
\[
\calL_1=\frac{\theta}{\dt}+\frac{\alpha}{2}-\frac{\mu}{2}\Delta_h,~~\calL_2=\frac{\theta}{\dt}-\frac{\alpha}{2}+\frac{\mu}{2}\Delta_h.
\]
In turn, the proposed numerical scheme could be rewritten as
\begin{equation}\label{Ma_PKS}
\left\{
\begin{aligned}
\frac{\rho^{n+1}-\rho^n}{\Delta t}=&\nabla_h\cdot\Big[ \frac{1}{f''(\hat{\rho}^{n+\frac{1}{2}})}\nabla_h \Big( \gamma S^{n+\frac{1}{2}} - \Big(\frac{\chi}{2}\calL_1^{-1}\calL_2\phi^{n}+\frac{\chi^2}{4}\calL_1^{-1}\rho^n+\frac{\chi}{2} \phi^n \\ 
 & \qquad\qquad\qquad\qquad
 +  \frac{\chi^2 \dt}{4\theta} \rho^n \Big)+\calG_h \rho^{n+1} \Big) \Big] ,\\
\calL_1\phi^{n+1}=&\calL_2\phi^n+\frac{\chi}{2}(\rho^{n+1}+\rho^n) ,
\end{aligned}
\right.
\end{equation}
with $\calG_h:= \frac{\chi^2 \dt}{4\theta}  -\frac{\chi^2}{4}\calL_1^{-1}$, and homogeneous boundary conditions are imposed.

\begin{remark}
A stabilization term $\frac{\chi^2 \dt}{4\theta+ 2\alpha \dt}  (\rho^{n+1}-\rho^n)$, instead of $\frac{\chi^2 \dt}{4\theta}  (\rho^{n+1}-\rho^n) $ in \reff{PKS2nd_1}, could be used in the numerical scheme for a small positive $\theta$, and the theoretical analysis on structure-preserving properties and convergence could still go through. 
\end{remark}

\section{Structure-preserving properties}
In this section, we prove the mass conservation, unique solvability, positivity-preserving properties of the second-order numerical scheme, as well as an unconditional dissipation of the original free energy functional, at the discrete level.

\begin{theorem}\label{th1}
{\bf (Mass conservation)} The second-order accurate numerical scheme (\ref{PKS2nd_1}-\ref{PKS2nd_2}) respects a discrete mass conservation law:
\[
\ciptwo{\rho^{n+1}}{1}=\ciptwo{\rho^{n}}{1}.
\]
\end{theorem}
Such a mass conservation identity is obtained by applying the summation on both sides, and the discrete homogeneous Neumann boundary conditions have been used.

The free energy is discretized as
\begin{equation}\label{EnergyKS1}
F_h(\rho^n,\phi^n)=\gamma\ciptwo{f(\rho^n)}{1}-\chi\ciptwo{\rho^n}{\phi^n}+\frac{\mu}{2}\| \nabla_h\phi^n\|^2_2+\frac{\alpha}{2}\| \phi^n\|^2_2,
\end{equation}
which turns out to be a second-order approximation to the continuous version of the energy.

 Meanwhile, the following monotonicity property is needed in the unique solvability analysis.

\begin{lemma}\label{lem:mono}
The linear operator $\calG_h$ satisfies the monotonicity condition:
\begin{equation} \label{lem:mono-0}
\ciptwo{\calG_h(\eta_1)-\calG_h(\eta_2)}{\eta_1-\eta_2}\geq 0 , ~~ \mbox{for }~\eta_1,\eta_2\in\calC.
\end{equation}
Furthermore, the equality is valid if and only if $\td\eta=0$, i.e., $\eta_1=\eta_2$, if $\overline{\eta_1}=\overline{\eta_2}=0$ is required. Therefore, the operator $\calG_h$ is invertible.
\end{lemma}

\begin{proof}
Denote a difference function $\td \eta=\eta_1-\eta_2$. Since $\calG_h$ is a linear operator, we have
\begin{equation}\label{lemo:eq1}
\calG_h(\eta_1)-\calG_h(\eta_2)=\calG_h(\td\eta)= \frac{\chi^2 \dt}{4\theta} \td\eta-\frac{\chi^2}{4}\calL_1^{-1}\td\eta.
\end{equation}
Taking a discrete inner product with \reff{lemo:eq1} by $\td \eta$ yields
\begin{equation}\label{lemo:eq2}
\ciptwo{\calG_h(\td \eta)}{\td \eta}= \frac{\chi^2 \dt}{4\theta} \|\td\eta \|^2_2-\frac{\chi^2}{4}\ciptwo{\calL^{-1}_1\td \eta}{\td \eta}.
\end{equation}
Based on the definition of $\calL_1^{-1}$, we see that 
\begin{equation}\label{lemo:eq4}
\frac{\chi^2}{4}\ciptwo{\td \eta}{\calL_1^{-1}\td\eta}\leq\frac{\chi^2\dt}{4\theta}\|\td\eta \|^2_2.
\end{equation}
Consequently, a combination of \reff{lemo:eq2} and \reff{lemo:eq4} leads to
\begin{equation}
\ciptwo{\calG_h(\td \eta)}{\td \eta}\geq 0 .
\end{equation}
In addition, the equality is valid if and only if $\td\eta=0$, i.e., $\eta_1=\eta_2$. The proof is complete.
\qed
\end{proof}

Moreover, a discrete maximum norm bound of the operator $\calL^{-1}$ is also needed in the later analysis. Using the technique presented in the Refs.~\cite{ChenWangWangWise_JCP2019, DingWangZhou_NMTMA19}, we state the following lemma without giving its proof.

\begin{lemma}\label{LemLf}
Assume that $\nu\in \mathring{\calC}$, $\|\nu \|_\infty\leq C_2$, and $f\in\calC$ satisfies $f\geq f_0>0$ (at a point-wise level). The following estimate is available:
\[
\|\calL^{-1}_f \nu \|_\infty\leq C_3 f_0^{-1}h^{-\frac{1}{2}},
\]
where $C_3>0$ only depends on $\Omega$ and $C_2$.
\end{lemma}

The positivity-preserving and unique solvability properties are proved in the following theorem. For simplicity of presentation, the classical PKS system, with $f (\rho) = \rho ( \ln \rho -1)$ (as given by \eqref{f_1}), is considered in the theoretical analysis. In turn, we get $f' (\rho) = \ln \rho$, $f'' (\rho) = \frac{1}{\rho}$ and $f''' (\rho) =- \frac{1}{\rho^2}$. An extension to the PKS system with a bounded mobility~\eqref{f_2} and a saturation density~\eqref{f_3} would be straightforward.
\begin{theorem}\label{t:EUP}
{\bf (Existence, uniqueness, and positivity-preserving property)}
 Define $C_{\rm min}^{n}:=\underset{1\leq i\leq N}{\min}\rho^{n}_{i}$, $C_{\rm max}^{n}:=\underset{1\leq i\leq N}{\max}\rho^{n}_{i}$ and $\| \phi^n\|_\infty\leq M^{n}_{\rm max}$. Given $C^{n}_{\rm min}>0$, there exists a unique solution to the second-order accurate scheme (\ref{PKS2nd_1}-\ref{PKS2nd_2}), such that
\begin{equation}\label{+rho}
\rho^{n+1}_{i,j,k}>0 , ~~\mbox{for}~~ i,j,k=1,2,\cdots,N.
\end{equation}
\end{theorem}

\begin{proof}
The numerical solution to the proposed algorithm (\ref{PKS2nd_1}-\ref{PKS2nd_2}) is equivalent to the minimizer of the discrete energy functional:
\begin{equation}\label{EnFun}
\begin{aligned}
\cJ^n(\rho)=&\frac{1}{2\Delta t} \|\rho-\rho^{n} \|^2_{\calL^{-1}_{\hat{\rho}^{n+\frac{1}{2}}}}+\gamma\ciptwo{\rho+\frac{5}{6}\rho^{n} }{\ln \rho -1} +\gamma\ciptwo{(\rho^{n})^2}{\frac{1}{6\rho}}-\frac{2}{3}\gamma\ciptwo{\rho}{1}\\
&-\frac{\chi^2}{8}\ciptwo{\rho}{\calL^{-1}_1\rho}
 + \frac{\chi^2 \dt}{8 \theta} \|\rho \|^2_2\\
&-\ciptwo{\frac{\chi}{2}\calL_1^{-1}\calL_2\phi^{n}+\frac{\chi^2}{4}\calL_1^{-1}\rho^n+\frac{\chi}{2}\phi^n+  \frac{\chi^2 \dt}{4\theta} \rho^n}{\rho} ,
\end{aligned}
\end{equation}
over the admissible set
\[
\K_{h}:=\left\{\rho\bigg| 0 <\rho_{i,j,k}< \xi,~ \frac{1}{|\Omega|}\ciptwo{\rho}{1}=Q^0 ,~~~i,j,k=1,\cdots,N \right\}, \, \, \,  Q^0 =\frac{1}{|\Omega|}\ciptwo{\rho^0}{1} ,  \, \, \,
 \xi:=\frac{Q^0 |\Omega|}{h^3} .
\]
Consider a closed subset $\K_{h,\delta}\subset\K_h$:
\[
\K_{h,\delta}:=\left\{\rho\bigg| \delta \leq \rho_{i,j,k}\leq \xi-\delta,~ \frac{1}{|\Omega|}\ciptwo{\rho}{1}=Q^0 ,~~~i,j,k=1,\cdots,N \right\}, \quad  \delta\in(0,\frac{\xi}{2}) .
\]
Obviously, $\K_{h,\delta}$ is a bounded, convex, and compact subset of $\K_h$. By the convexity of $\cJ^n$, there exists a unique minimizer of $\cJ^n$ in $\K_{h,\delta}$.

 Suppose that the minimizer of $\cJ^n$, $\rho^*$, touches the boundary of $\K_{h,\delta}$. Assume that there exists a grid point $\vec{\alpha}_0=(i_0,j_0,k_0)$ such that $\rho^{*}_{\vec{\alpha}_0}=\delta$, and a grid point $\vec{\alpha}_1=(i_1,j_1,k_1)$ such that the maximum of $\rho^{*}$ is achieved. It is clear that the maximum value $\rho^{*}_{\vec{\alpha}_1}$ is larger than the mean value $Q^0$, and the minimum value $\rho^{*}_{\vec{\alpha}_0}$ is less than $Q^0$, i.e.,
\[
\rho^{*}_{\vec{\alpha}_1}\geq Q^0 , ~~
\rho^{*}_{\vec{\alpha}_0}\leq Q^0 .
\]

Consider the following directional derivative 
\[
\begin{aligned}
\lim_{t\rightarrow 0^+}&\frac{\cJ^n(\rho^*+td)-\cJ^n(\rho^*)}{t}\\
=&\frac{1}{\Delta t}\ciptwo{\calL^{-1}_{\hat{\rho}^{n+\frac{1}{2}}} (\rho^{*}-\rho^{n})}{d} +\gamma\ciptwo{\ln \rho^{*}+\frac{5\rho^{n}}{6\rho^{*}} -\frac{(\rho^{n})^2}{6(\rho^{*})^2}-\frac{2}{3}}{d}\\
&+\ciptwo{\calG_h\rho^*}{d}-\ciptwo{\frac{\chi}{2}\calL_1^{-1}\calL_2\phi^{n}+\frac{\chi^2}{4}\calL_1^{-1}\rho^n+\frac{\chi}{2}\phi^n+ \frac{\chi^2 \dt}{4\theta} \rho^n}{d},
\end{aligned}
\]
with the direction 
\[
d=\delta_{i,i_0}\delta_{j,j_0}\delta_{k,k_0}-\delta_{i,i_1}\delta_{j,j_1}\delta_{k,k_1},
\]
where $\delta_{l,k}$ is the Kronecker symbol. Clearly, $d\in\mathring{\calC}$. In turn, the directional derivative becomes
\begin{equation}\label{thp:eq1}
\begin{aligned} 
 & 
\frac{1}{h^3}\lim_{t\rightarrow 0^+} \frac{\cJ^n(\rho^*+td)-\cJ^n(\rho^*)}{t}\\
=&\frac{1}{\Delta t} \calL^{-1}_{\hat{\rho}^{n+\frac{1}{2}}}(\rho^{*}-\rho^{n})_{\vec{\alpha}_0}-\frac{1}{\Delta t}\calL^{-1}_{\hat{\rho}^{n+\frac{1}{2}}}( \rho^{*}-\rho^{n})_{\vec{\alpha}_1}+(\calG_h\rho^*)_{\vec{\alpha}_0}-(\calG_h\rho^*)_{\vec{\alpha}_1}\\
&+\gamma\Big(\ln \rho^{*}+\frac{5\rho^{n}}{6\rho^{*}} -\frac{(\rho^{n})^2}{6(\rho^{*})^2}\Big)_{\vec{\alpha}_0}-\gamma\Big(\ln \rho^{*} +\frac{5\rho^{n}}{6\rho^{*}} -\frac{(\rho^{n})^2}{6(\rho^{*})^2}\Big)_{\vec{\alpha}_1}\\
&-\Big(\frac{\chi}{2}\calL_1^{-1}\calL_2\phi^{n}+\frac{\chi^2}{4}\calL_1^{-1}\rho^n+\frac{\chi}{2}\phi^n+  \frac{\chi^2 \dt}{4\theta} \rho^n \Big)_{\vec{\alpha}_0} \\
&
+\Big( \frac{\chi}{2}\calL_1^{-1}\calL_2\phi^{n}+\frac{\chi^2}{4}\calL_1^{-1}\rho^n+\frac{\chi}{2}\phi^n+  \frac{\chi^2 \dt}{4\theta}  \rho^n \Big)_{\vec{\alpha}_1}.
\end{aligned}
\end{equation}
Define $C^{n}_{\rm max}=\underset{i,j,k=1,\cdots,N}{\max} \rho^{n}_{i,j,k}$ and $C^{n}_{\rm min}=\underset{i,j,k=1,\cdots,N}{\min} \rho^{n}_{i,j,k}$.
Since $\rho^{*}_{\vec{\alpha}_0}=\delta$ and $\rho^{*}_{\vec{\alpha}_1}\geq Q$, we have
\begin{equation}\label{thp:eq2}
\begin{aligned}
\gamma\left[\ln (\rho^{*})+\frac{5\rho^{n}}{6\rho^{*}} -\frac{(\rho^{n})^2}{6(\rho^{*})^2} \right]_{\vec{\alpha}_0}&-\gamma\left[\ln (\rho^{*})+\frac{5\rho^{n}}{6\rho^{*}} -\frac{(\rho^{n})^2}{6(\rho^{*})^2} \right]_{\vec{\alpha}_1}\\
\leq& \gamma\ln \delta +\frac{5\gamma C^{n}_{\rm max}}{6\delta}-\frac{\gamma(C^{n}_{\rm min})^2}{6\delta^2}-\gamma\ln Q^0 +\frac{\gamma(C^{n}_{\rm max})^2}{6( Q^0 )^2}.
\end{aligned}
\end{equation}
With a similar analysis technique as in Lemma~\ref{LemLf}, one can derive that $\|\calG_h\rho^*\|_\infty \leq M_1/2$ with $\|\rho^{*}\|_{\infty} \leq \xi$, where $M_1$ is a constant independent of $\delta$. Then we have 
\begin{equation}\label{thp:eq3}
(\calG_h\rho^*)_{\vec{\alpha}_0}-(\calG_h\rho^*)_{\vec{\alpha}_1} \leq M_1.
\end{equation}
On the other hand, the operator $\frac{\chi}{2}\calL_1^{-1}\calL_2\phi^{n}+\frac{\chi^2}{4}\calL_1^{-1}\rho^n+ \frac{\chi}{2}\phi^n+ \frac{\chi^2 \dt}{4\theta} \rho^n$ is linear with respect to $\phi^n$ and $\rho^n$. Subsequently, the a-priori assumptions $\| \rho^n\|_\infty\leq C^{n}_{\rm max}$ and $\| \phi^n\|_\infty\leq M^{n}_{\rm max}$ indicate that
\begin{equation}\label{thp:eq4}
\begin{aligned}
&\left(\frac{\chi}{2}\calL_1^{-1}\calL_2\phi^{n}+\frac{\chi^2}{4}\calL_1^{-1}\rho^n+\frac{\chi}{2}\phi^n+  \frac{\chi^2}{4\theta} \rho^n\right)_{\vec{\alpha}_1}\\
&\qquad-\left(\frac{\chi}{2}\calL_1^{-1}\calL_2\phi^{n}+\frac{\chi^2}{4}\calL_1^{-1}\rho^n+\frac{\chi}{2}\phi^n+ \frac{\chi^2}{4\theta} \rho^n\right)_{\vec{\alpha}_0}\leq M_2,
\end{aligned}
\end{equation}
where $M_2$ is another constant independent of $\delta$. By the bound $\|\rho^{*}-\rho^{n} \|_{\infty}\leq \xi+ C^{n}_{\rm max} $, we get
\begin{equation}\label{thp:eq5}
\calL^{-1}_{\hat{\rho}^{n+\frac{1}{2}}}(\rho^{*}-\rho^{n})_{\vec{\alpha}_0}-\calL^{-1}_{\hat{\rho}^{n+\frac{1}{2}}}( \rho^{*}-\rho^{n})_{\vec{\alpha}_1}\leq 2 M_3,
\end{equation}
where $M_3$ is a constant dependent on $C^{n}_{\rm max} $, $\dt$, $h$, $\Omega$, $\xi$, and Lemma \ref{LemLf} has been applied.

Substituting \reff{thp:eq2},  \reff{thp:eq3}, \reff{thp:eq4}, and \reff{thp:eq5} into \reff{thp:eq1}, we obtain
\begin{equation}\label{thp:eq7}
\begin{aligned}
\frac{1}{h^3}\lim_{t\rightarrow 0^+}\frac{\cJ^n(\rho^*+td)-\cJ^n(\rho^*)}{t}\leq& 2(\Delta t)^{-1}M_3+\gamma\ln \delta +\frac{5\gamma C^{n}_{\rm max}}{6\delta}+M_1+M_2\\
&-\frac{\gamma(C^{n}_{\rm min})^2}{6\delta^2}-\gamma\ln Q^0 +\frac{\gamma(C^{n}_{\rm max})^2}{6 (Q^0 )^2}.
\end{aligned}
\end{equation}
For any fixed $\Delta t$ and $h$, the value of $\delta$ could be chosen sufficiently small so that
\begin{equation}\label{thp:eq11}
\begin{aligned}
2(\Delta t)^{-1}M_3+&\gamma\ln \delta +\frac{5\gamma C^{n}_{\rm max}}{6\delta}
-\frac{\gamma(C^{n}_{\rm min})^2}{6\delta^2}-\gamma\ln Q^0 +\frac{\gamma(C^{n}_{\rm max})^2}{6 (Q^0 ) ^2}+M_1+M_2<0.
\end{aligned}
\end{equation}
Therefore, the following inequality is valid:
\begin{equation}\label{thp:eq12}
\lim_{t\rightarrow 0^+}\frac{\cJ^n(\rho^*+td)-\cJ^n(\rho^*)}{t}<0.
\end{equation}
This is contradictory to the assumption that $\rho^{*}$ is the minimizer of $\cJ^n$.

Similarly, we are able to prove that the minimizer of $\cJ^n$ cannot occur at the upper boundary of $\K_{h,\delta}$. In fact, if this occurs, there must be a grid point, at which the value of $\rho^{*}$ approaches zero. A contradiction could be obtained in the same manner as above. Therefore, the global minimum of $\cJ^n$ could only possibly achieve at an interior point, i.e., $\rho^*\in \mathring{\K}_{h,\delta}\subset \mathring{\K}_{h}$ as $\delta\rightarrow 0$. Since $\cJ^n$ is a smooth functional, there must exist a solution $\rho^*\in \mathring{\K}_{h,\delta}\subset \mathring{\K}_{h}$, satisfying
\begin{equation}\label{thp:eq13}
\lim_{t\rightarrow 0^+}\frac{\cJ^n(\rho^*+td)-\cJ^n(\rho^*)}{t}=0 .
\end{equation}
As a result, there exists a positive numerical solution $\rho^*$ to the numerical system (\ref{PKS2nd_1}-\ref{PKS2nd_2}). The uniqueness of the numerical solution is a direct consequence of the strict convexity of the discrete energy functional $\cJ^n(\rho)$.
\qed
\end{proof}

\begin{theorem} \label{thm: energy stability}
{\bf (The original energy dissipation)}
The second-order numerical scheme (\ref{PKS2nd_1}-\ref{PKS2nd_2}) respects a dissipation law of the discrete free energy \reff{EnergyKS1}:
\begin{equation}\label{KSED:eq0}
\begin{aligned}
F^{n+1}_h-F^n_h\leq&-\Delta t[\frac{1}{f''(\hat{\rho}^{n+\frac{1}{2}})}\nabla_h v^{n+\frac{1}{2}},\nabla_h v^{n+\frac{1}{2}}]-\frac{\theta}{\Delta t}\|\phi^{n+1}-\phi^n\|^2_2 
-  \frac{\chi^2 \dt}{4\theta} \|\rho^{n+1}-\rho^n \|^2_2\leq 0,
\end{aligned}
\end{equation}
with $v^{n+\frac{1}{2}}=\gamma S^{n+\frac{1}{2}}-\frac{\chi}{2}(\phi^{n+1}+\phi^n)+\frac{\chi^2\dt}{4\theta} (\rho^{n+1}-\rho^n) $.
\end{theorem}
\begin{proof}
Taking a discrete inner product with \reff{PKS2nd_1} by $\Delta t v^{n+\frac{1}{2}}$,
we get
\begin{equation}\label{KSED:eq1}
\ciptwo{\rho^{n+1}-\rho^n}{v^{n+\frac{1}{2}}}=-\Delta t[\frac{1}{f''(\hat{\rho}^{n+\frac{1}{2}})}\nabla_h v^{n+\frac{1}{2}},\nabla_h v^{n+\frac{1}{2}}].
\end{equation}
For any function $H(\cdot)\in C^4(\R)$, the following Taylor expansion is valid:
\[
\begin{aligned}
H(x)=&H(y)+H^{(1)}(y)(x-y)+\frac{1}{2}H^{(2)}(y)(x-y)^2\\
&+\frac{1}{6}H^{(3)}(y)(x-y)^3+\frac{1}{24}H^{(4)}(\eta)(x-y)^4,~ \forall x, y \in \R,
\end{aligned}
\]
where $\eta$ is between $x$ and $y$, and $H^{(p)}(y)=\frac{\partial^{p}H}{\partial y^p}$ for $p=1,2,3,4$. If $H^{(4)}(\eta)>0$, one has
\[
H(y)-H(x)\leq \left(H^{(1)}(y)-\frac{1}{2}H^{(2)}(y)(y-x)+\frac{1}{6}H^{(3)}(y)(y-x)^2\right)(y-x).
\]
Choosing $H(\rho)=f(\rho)$, we have
\begin{equation}\label{KSED:eq2}
\begin{aligned}
\ciptwo{\rho^{n+1}-\rho^n}{v^{n+\frac{1}{2}}}\geq& \gamma\ciptwo{f(\rho^{n+1})-f(\rho^n)}{1}-\frac{\chi}{2}\ciptwo{\rho^{n+1}-\rho^n}{\phi^{n+1}+\phi^n}\\
&+  \frac{\chi^2 \dt}{4\theta} \|\rho^{n+1}-\rho^n \|^2_2.
\end{aligned}
\end{equation}
On the other hand, taking a discrete inner product with \reff{PKS2nd_2} by $-(\phi^{n+1}-\phi^n)$, we have
\begin{equation}\label{KSED:eq3}
\begin{aligned}
-\frac{\theta}{\Delta t}\|\phi^{n+1}-\phi^n\|^2_2=&\frac{\mu}{2}\left(\|\nabla_h\phi^{n+1}\|^2-\|\nabla_h\phi^{n}\|^2\right)\\
&+\frac{\alpha}{2}\left(\|\phi^{n+1}\|^2-\|\phi^{n}\|^2\right)-\frac{\chi}{2}\ciptwo{\rho^n+\rho^{n+1}}{\phi^{n+1}-\phi^n}.
\end{aligned}
\end{equation}
Moreover, the following equality is valid:
\begin{equation}\label{KSED:eq4}
-\frac{\chi}{2}\ciptwo{\rho^{n+1}-\rho^n}{\phi^{n+1}+\phi^n}-\frac{\chi}{2}\ciptwo{\rho^n+\rho^{n+1}}{\phi^{n+1}-\phi^n}=-\chi\left(\ciptwo{\rho^{n+1}}{\phi^{n+1}}-\ciptwo{\rho^{n}}{\phi^{n}}\right).
\end{equation}
A combination of \reff{KSED:eq1}-\reff{KSED:eq4} leads to the energy dissipation inequality~\reff{KSED:eq0}. \qed
\end{proof}

\begin{remark}
Structure-preserving properties, such as the unique solvability of positive numerical solution~\reff{+rho} and the original energy dissipation~\reff{KSED:eq0}, can be analogously proved for the PKS system with a bounded mobility~\eqref{f_2} or a saturation density~\eqref{f_3}.
\end{remark}

\begin{remark}
If the free energy contains both the convex and concave parts, some existing works have reported a modified energy stability analysis for various second order accurate, multi-step numerical schemes~\cite{ChenJingWangWangWise_CiCP22, DongWangZhangZhang_CiCP20, HuangShen_2021SISC}. However, these reported stability analysis is in terms of a modified discrete energy. In comparison, the stability estimate~\eqref{KSED:eq0} is in terms of the original free energy \reff{EnergyKS1}, which turns out to be a remarkable theoretical result that has been rarely reported. 
\end{remark}

\section{Convergence analysis}\label{s:convergence}
Let $(\phi_e,\rho_e)$ be the exact solution to the PKS system \reff{KSeqn}. The following regularity assumption is made:
\[
\phi_e,\rho_e\in\mathcal{R}:=H^6(0,T;C(\Omega))\cap H^5(0,T;C^2(\Omega))\cap L^\infty(0,T;C^6(\Omega)).
\]
In addition, the following separation property is assumed for the exact solution, for the convenience of the analysis:
\[
\rho_e\geq \ve^*, \mbox{ for some }\ve^*>0,~ \mbox{at a point-wise level}.
\]

Define $\rho_N:=\mathcal{P}_N \rho_e(\cdot,t)$ and $\phi_N:=\mathcal{P}_N \phi_e(\cdot,t)$ as the Fourier Cosine projection of the exact solution into $\mathcal{P}^K$, which is the space of trigonometric polynomials in $x$, $y$, and $z$ of degree up to $K$ (with $K=N-1$). See more details in our previous work~\cite{DingWangZhou_NMTMA19}. The following projection approximation is standard, for $(\phi_e,\rho_e)\in L^\infty(0,T;H^m(\Omega))$, with $m\in\mathbb{N}$, $0\leq k\leq m$: 
\begin{equation}
\begin{aligned}
\|\phi_N-\phi_e \|_{L^{\infty}(0,T;H^k)}&\leq Ch^{m-k}\| \phi_e\|_{L^{\infty}(0,T;H^m)},\\
\|\rho_N-\rho_e \|_{L^{\infty}(0,T;H^k)}&\leq Ch^{m-k}\| \rho_e\|_{L^{\infty}(0,T;H^m)}.
\end{aligned}
  \label{projection-ets-1}
\end{equation}
Notice that the Fourier Cosine projection estimate does not preserve the positivity of the variables, while we are always able to take $h$ sufficiently small (corresponding to a large $N$) so that
\[
\rho_N\geq \frac{1}{2}\ve^*.
\]
Denote by $\phi^n_N=\phi_N(\cdot,t_n)$ and $\rho^{n}_N=\rho_N(\cdot,t_n)$, with $t_n=n\dt$. Since $\rho_N^n \in \mathcal{P}^K$, the mass conservative property is available at the discrete level:
\[
\overline{\rho^{n}}=\frac{1}{|\Omega|}\int_{\Omega} \rho(\cdot,t_n)d\x=\frac{1}{|\Omega|}\int_{\Omega} \rho(\cdot,t_{n-1})d\x=\overline{\rho^{n-1}}~~\mbox{for}~~n\in\N^*.
\]
On the other hand, the numerical solution of the second-order scheme (\ref{PKS2nd_1}-\ref{PKS2nd_2}) is also mass conservative at the discrete level:
\[
\overline{\rho^{n-1}}=\overline{\rho^{n}}~~\mbox{for}~~n\in\N^*.
\]
In turn, the mass conservative projection is made for the initial data:
 \[
 \begin{aligned}
 &\phi^{0}=\mathcal{P}_h \phi_N(\cdot,t=0):=\phi_N(p_i,p_j,p_k,t=0),\\
 &\rho^{0}=\mathcal{P}_h \rho_N(\cdot,t=0):=\rho_N(p_i,p_j,p_k,t=0).
 \end{aligned}
\]
Accordingly, the error grid functions are defined as
\begin{equation}\label{ca:eq1}
e^{n}_{\phi}:=\mathcal{P}_h \phi^{n}_N-\phi^{n},~~e^{n}_{\rho}:=\mathcal{P}_h \rho^{n}_N-\rho^{n} ,  \quad  n\in\mathbb{N}^*.
\end{equation}
As indicated above, one can verify that $\bar{e}^{n}_{\phi}=0$,  $\bar{e}^{n}_{\rho}=0$, for $n\in\mathbb{N}$, so that the discrete norm $\|\cdot \|_{-1,h}$ is well defined for the error grid functions.

The following theorem is the main result of this section.
\begin{theorem}\label{Th:convergence}
Given initial data $\phi_e(\cdot,t=0),\rho_e(\cdot,t=0) \in C^6(\Omega)$, suppose the exact solution for the PKS system \reff{KSeqn} is of regularity class $\mathcal{R}$. Let $e^{n}_{\phi}$ and $e^{n}_{\rho}$ be the error grid functions defined in~\reff{ca:eq1}.  Then, under the linear refinement requirement $\lambda_1h\leq \Delta t\leq \lambda_2h$, the following convergence result is available as $\dt, h \to 0$:
\begin{equation}\label{th:eq1}
\begin{aligned}
  &
 \big\|e^{n}_{\rho} \big\|_2 + \Big(\dt\sum_{k=0}^{n-1}\|\nabla_h ( e^{k}_{\rho} + e^{k+1}_\rho ) \|_2^2 \Big)^{\frac{1}{2}}
\\
  &
   \qquad+ \|e^{n}_{\phi} \|_{2} + \| \nabla_h e^n_\phi \|_2
 + \Big( \dt\sum_{k=0}^{n-1}\|\Delta_h ( e^k_{\phi} + e^{k+1}_\phi ) \|_2^2 \Big)^{\frac{1}{2}}
  \leq C(\Delta t^2+h^2),~n\in\N,
\end{aligned}
\end{equation}
where $t_{n}=n\dt\leq T$ and the constant $C>0$ is independent of $\Delta t$ and $h$.
\end{theorem}
\vskip2mm

\subsection{Higher-order consistency analysis}
The leading local truncation error will not be sufficient to recover an $\ell^\infty$ bound of the discrete temporal derivative of the numerical solution, which is needed in the nonlinear convergence analysis. To overcome this subtle difficulty, we apply a higher order consistency estimate via a perturbation analysis~\cite{LiuWangWiseYueZhou_2021, LiuWang_JSC2023}. Such a higher order consistency result is stated below, and the detailed proof follows a similar idea as in~\cite{LiuWang_JSC2023}. The technical details are skipped for the sake of brevity.

\begin{proposition}
Let $(\phi_e,\rho_e)$ be the exact solution to the PKS system~\reff{PKSeqns} and $(\phi_N, \rho_N )$ be its Fourier Cosine projection.
There exists auxiliary variables, $\phi_{\dt,1}$, $\phi_{\dt,2}$, $\phi_{h,1}$, $\rho_{\dt,1}$, $\rho_{\dt,2}$, $\rho_{h,1}$, so that the following expansion profiles
\begin{equation} \label{Consis:eq0}
\begin{aligned}
&\check{\phi} =\phi_N+ \mathcal{P}_N\left( \dt^2  \phi_{\dt,1} +\dt^3  \phi_{\dt,2} +h^2  \phi_{h,1}\right), \\
&\check{\rho} =\rho_N+\mathcal{P}_N\left( \dt^2  \rho_{\dt,1} +\dt^3  \rho_{\dt,2} +h^2  \rho_{h,1}\right) ,
\end{aligned}
\end{equation}
satisfy the numerical scheme up to an $O(\dt^4+h^4)$ consistency:
\begin{equation}\label{Consis:eq1}
\begin{aligned}
\frac{\crho^{n+1}-\crho^{n}}{\dt}=&\nabla_h\cdot\Big[ \Big( \frac{3}{2}\crho^{n}-\frac{1}{2}\crho^{n-1} \Big)\nabla_h \Big(\gamma \check{S}^{n+\frac{1}{2}}- \frac{\chi}{2}(\cphi^n+\cphi^{n+1}) \\
  & \qquad\qquad\qquad\qquad\quad
 +\frac{\chi^2\dt}{4\theta} (\crho^{n+1}-\crho^n) \Big) \Big]+\tau^{n+\frac{1}{2}}_\rho,\\
 \check{S}^{n+\frac{1}{2}}= & \ln (\crho^{n+1}) - \frac{1}{2 \crho^{n+1}} (\crho^{n+1}-\crho^n) -\frac{1}{6 ( \crho^{n+1})^2 } (\crho^{n+1}-\crho^n)^2 , \\
\theta\frac{\cphi^{n+1}-\cphi^n}{\dt}=&\frac{\mu}{2} \Delta_h ( \cphi^{n+1}+ \cphi^n)-\frac{\alpha}{2}(\cphi^{n+1}+\cphi^n)+\frac{\chi}{2}(\crho^{n+1}+\crho^n)+\tau^{n+\frac{1}{2}}_\phi,
\end{aligned}
\end{equation}
with $\|\tau^{n+\frac{1}{2}}_\rho\|_2,\|\tau^{n+\frac{1}{2}}_\phi\|_2\leq C(\dt^4+h^4)$.
The constructed variables $\phi_{\dt,1}$, $\phi_{\dt,2}$, $\phi_{h,1}$, $\rho_{\dt,1}$, $\rho_{\dt,2}$, $\rho_{h,1}$ solely depend on the exact solution $(\phi_e,\rho_e)$, and their derivatives are bounded.

(1) The following mass conservative identities and zero-mean property for the local truncation error are available:
\begin{equation}\label{ConstSols:eq0}
\begin{aligned}
&\rho^{0}\equiv\crho^{0},~~\overline{\rho^{n}}=\overline{\rho^{0}},~~n\in \mathbb{N},\\
&\overline{\crho^{n}}=\frac{1}{|\Omega|}\int_{\Omega} \crho(\cdot,t_n)d\x=\frac{1}{|\Omega|}\int_{\Omega} \crho^{0}d\x=\overline{\crho^{0}},~~n\in \mathbb{N},\\
&\overline{\tau^{n+\frac{1}{2}}_\rho}=0,~~n\in \mathbb{N}.
\end{aligned}
\end{equation}

(2) A similar phase separation property is valid for the constructed $\crho$, for some $\ve^* >0$:
\begin{equation}\label{ConstSols:eq1}
\crho\geq \ve^* >0.
\end{equation}

(3) A discrete $W^{1,\infty}$ bound for the constructed profile $\crho$, as well as its discrete temporal derivative, is available at any time step $t^k$:
\begin{equation} \label{ConstSols:eq2}
 \| \crho^k \|_{\infty}\leq C^*, \quad
 \| \nabla_h\crho^k \|_{\infty}\leq C^* ,  \quad
 \| \crho^k - \crho^{k-1} \|_\infty \le C^* \dt , \quad
 \| \nabla_h ( \crho^k - \crho^{k-1} ) \|_\infty \le C^* \dt .
\end{equation}

\end{proposition}

\subsection{A rough error estimate}
Instead of analyzing the original numerical error functions defined in~\eqref{ca:eq1}, we consider the following ones
\begin{equation}\label{Higherr}
\td{\phi}^n:=\mathcal{P}_h \cphi^n-\phi^n,~~\td{\rho}^{n}:=\mathcal{P}_h \crho^{n}-\rho^{n} , \quad  n\in \mathbb{N}.
\end{equation}
For the convenience of the notation, the following average numerical error functions are introduced at the intermediate time instant $t_{n+\frac{1}{2}}$:
\begin{equation*}
\begin{aligned}
  &
\check{\rho}^{n+\frac{1}{2}}=\frac{3}{2}\check{\rho}^{n}-\frac{1}{2}\check{\rho}^{n-1} ,
\\
  &
\td{\hat{\rho}}^{n+\frac{1}{2}}=\check{\rho}^{n+\frac{1}{2}}-\hat{\rho}^{n+\frac{1}{2}}
= \frac32 \crho^n - \frac12 \crho^{n-1} - \Big( ( \frac32 \rho^n - \frac12 \rho^{n-1} )^2
+ \dt^8 \Big)^\frac12 .
\end{aligned}
\end{equation*}

Subtracting the numerical scheme (\ref{PKS2nd_1}-\ref{PKS2nd_2}) from the consistency estimate \reff{Consis:eq1} yields
\begin{align}
&\frac{\td{\rho}^{n+1}-\td{\rho}^{n}}{\Delta t}=\nabla_h\cdot\left(\hat{\rho}^{n+\frac{1}{2}}\nabla_h \td v^{n+\frac{1}{2}}+\td{\hat{\rho}}^{n+\frac{1}{2}}\nabla_h \calV^{n+\frac{1}{2}} \right)+\tau^{n+\frac{1}{2}}_{\rho},\label{ErrFun:eq0-1}\\
&\theta\frac{\td\phi^{n+1}-\td\phi^n}{\Delta t}=\frac{\mu}{2}\Delta_h(\td\phi^n+\td\phi^{n+1})-\frac{\alpha}{2}(\td\phi^{n+1}+\td\phi^n)+\frac{\chi}{2}(\td\rho^{n+1}+\td\rho^n)+\tau^{n+\frac{1}{2}}_\phi,\label{ErrFun:eq0-2}
\end{align}
where
\begin{equation}\label{ErrFun:eq1}
\begin{aligned}
\td v^{n+\frac{1}{2}}=&\gamma \td S^{n+\frac{1}{2}} - \frac{\chi}{2}(\tilde{\phi}^{n+1} + \tilde{\phi}^n)
  + \frac{\chi^2 \dt}{4 \theta}  (\tilde{\rho}^{n+1} - \tilde{\rho}^n) ,\\
 \td{S}^{n+\frac12} = & \ln (\crho^{n+1}) - \ln (\rho^{n+1}) - \frac{1}{2 \rho^{n+1}}
 ( \tilde{\rho}^{n+1}- \tilde{\rho}^n) + \frac{\tilde{\rho}^{n+1}}{2 \crho^{n+1} \rho^{n+1}} (\crho^{n+1} - \crho^n) \\
   &   - \frac{\crho^{n+1} - \crho^n + \rho^{n+1}-\rho^n}{6 ( \rho^{n+1})^2 }
   (\tilde{\rho}^{n+1} - \tilde{\rho}^n)
    + \frac{( \crho^{n+1} + \rho^{n+1} ) \tilde{\rho}^{n+1} }{6 ( \crho^{n+1})^2 ( \rho^{n+1} )^2 }
      (\crho^{n+1}-\crho^n)^2 , \\
\calV^{n+\frac{1}{2}}=&\gamma \check{S}^{n+\frac{1}{2}}-\frac{\chi}{2}(\cphi^{n+1}+\cphi^n)+\frac{\chi^2\dt}{4\theta}( \crho^{n+1}-\crho^n).\\
\end{aligned}
\end{equation}

A discrete $W_h^{1,\infty}$ bound could be assumed for $\calV^{n+\frac{1}{2}}$, due to the fact that it only depends on the exact solution and the constructed profiles:
\begin{equation}\label{A-priAssum:eq0}
\| \calV^{n+\frac{1}{2}}\|_{W^{1,\infty}_h}\leq C^*.
\end{equation}
In addition, we make the following a-prior assumption at the previous time steps:
\begin{equation}\label{A-priAssum:eq1}
\| \td \rho^{k}\|_2, \, \, \, \| \td \phi^{k}\|_2\leq \dt^{\frac{15}{4}}+h^{\frac{15}{4}}, \quad
   k=n, n-1,n-2.
\end{equation}
Such an a-priori assumption will be recovered by the optimal rate convergence analysis at the next time step, as will be proved later. Thanks to the inverse inequality, the $W_h^{1, \infty}$ bound for the numerical error function is available at the previous time steps:
\begin{equation}\label{presolns:eq1}
\begin{aligned}
&\| \td \rho^{k}\|_\infty\leq \frac{C\|\td \rho^{k} \|_2}{h^\frac{3}{2}}\leq \frac{C(\dt^{\frac{15}{4}}+h^{\frac{15}{4}})}{h^\frac{3}{2}}\leq C(\dt^{\frac{9}{4}}+h^{\frac{9}{4}})\leq \frac{\ve^*}{2},\\
&\| \nabla_h\td \rho^{k}\|_\infty\leq \frac{2 \| \td \rho^{k} \|_\infty}{h}\leq\frac{C(\dt^{\frac{9}{4}}+h^{\frac{9}{4}})}{h}\leq C(\dt^{\frac{5}{4}}+h^{\frac54})\leq 1,
\end{aligned}
\end{equation}
where the linear refinement constraint $\lambda_1 h\leq \dt\leq\lambda_2 h$ has been used. Subsequently, combined with the regularity assumption \reff{ConstSols:eq2}, a $W_h^{1, \infty}$ bound for the numerical solution could be derived at the previous time steps:
\begin{equation}\label{presolns:eq2}
\begin{aligned}
&\|  \rho^{k}\|_\infty\leq \| \crho^{k}\|_\infty+\| \td \rho^{k}\|_\infty\leq C^*+\frac{\ve^*}{2}:=\ckC_0,~~k=n, n-1,n-2,\\
&\|  \nabla_h\rho^{k}\|_\infty\leq \| \nabla_h\crho^{k}\|_\infty+\| \nabla_h\td \rho^{k}\|_\infty\leq C^*+1 := \tilde{C}_0 .
\end{aligned}
\end{equation}
Its combination with the separation estimate for $\crho$ results in a similar separation property for the numerical solution at the previous time steps:
\begin{equation}\label{presolns:eq3}
\rho^{k}\geq \crho^{k}-\| \td \rho^{k}\|_\infty\geq \frac{\ve^*}{2},~~k=n, n-1,n-2.
\end{equation}
Moreover, the discrete temporal derivative of the numerical solution at the previous time steps has to be bounded, for $k=n, n-1, n-2$, and such a bound will be useful in the later analysis:
\begin{equation}
\begin{aligned}
  &
  \| \tilde{\rho}^k - \tilde{\rho}^{k-1} \|_\infty \le  \| \tilde{\rho}^k \|_\infty + \| \tilde{\rho}^{k-1} \|_\infty
  \le C (\dt^\frac94 + h^\frac94)  \le \dt ,
\\
  &
  \| \rho^k - \rho^{k-1} \|_\infty \le \| \crho^k - \crho^{k-1} \|_\infty
  + \| \tilde{\rho}^k - \tilde{\rho}^{k-1} \|_\infty  \le ( C^* + 1 ) \dt = \tilde{C}_0 \dt ,  \, \, \,
  \mbox{(by~\eqref{ConstSols:eq2})} .
\end{aligned}
  \label{presolns:eq4}
\end{equation}

The following preliminary estimate will be used in the later analysis; its proof is based on direct calculations. The details are left to interested readers.
\begin{lemma}
The following bounds are valid at the intermediate time instant $t_{n+\frac{1}{2}}$:
\begin{equation}\label{lemim}
\begin{aligned}
& \frac{\ve^*}{2} \le \hat{\rho}^{n+\frac12} \le \check{C}_0 , \quad
 \| \td{\hat{\rho}}^{n+\frac{1}{2}} \|_2\leq \frac32 \|\td \rho^{n} \|_2
 + \frac12 \| \td \rho^{n-1}\|_2 + \dt^4 ,
 \\
   &
   \| \hat{\rho}^{n+\frac12} - \hat{\rho}^{n-\frac12} \|_\infty \le \frac32 \| \rho^n - \rho^{n-1} \|_\infty +
   \frac12 \| \rho^{n-1} - \rho^{n-2} \|_\infty + 2 \dt^4 \le 2 \td C_0 \dt .
\end{aligned}
\end{equation}
\end{lemma}

Before proceeding into the error estimate, a rough bound control of the nonlinear error inner products, namely, $\langle \td{\rho}^{n+1} , \gamma \td S^{n+\frac{1}{2}} \rangle$, is necessary. A preliminary estimate is stated in the following lemma; the detailed proof is  provided in Appendix \ref{Ap:A}.
\begin{lemma}\label{Lem:CA1}
Suppose the assumptions of the regularity requirement \reff{ConstSols:eq2}, phase separation \reff{ConstSols:eq1} for the constructed approximate solution $(\cphi,\crho)$, and the a-priori assumption \reff{A-priAssum:eq1} hold. In addition, let $\td \psi^n$ be an another error function with $\| \td \psi^n \|_\infty \le h$. Define the following set
\begin{equation}\label{Lemr:eq0}
\K=\left\{(i,j,k):~\rho_{i,j,k}\geq 2C^*+1 \right\},
\end{equation}
and denote $L^*:=| \K|$, the number of grid points in $\K$. Then there exists a constant $\ckC_2$ dependent only on $\ve^*$,  $\gamma$, $\ckC_0$ and $C^*$ such that
\begin{equation}\label{Lemr:eq1}
\langle \td{\rho}^{n+1} , \gamma \td S^{n+\frac{1}{2}} + \td \psi^n \rangle
\geq \frac{C^*}{6} \gamma L^* h^3-\ckC_2 ( \gamma^2 \|\td \rho^{n} \|^2_2 + \| \td \psi^n \|_2^2 ) . 
\end{equation}
In addition, if $L^*=0$, i.e., $\K$ is an empty set, there exists a constant $\ckC_3$ dependent on $C^*$ and $\gamma$ such that
\begin{equation}\label{Lemr:eq2}
\langle \td{\rho}^{n+1} , \gamma \td S^{n+\frac{1}{2}} + \td \psi^n \rangle
  \geq \ckC_3\|\td \rho^{n+1} \|^2_2-\ckC_2 ( \gamma^2\|\td \rho^{n} \|^2_2 + \| \td \psi^n \|_2^2 )  .
\end{equation}
\end{lemma}

The following proposition states a rough error estimate.
\begin{proposition} \label{prop:rough}
Based on the regularity requirement assumption \reff{A-priAssum:eq0} for the constructed profile $\calV^{n+\frac{1}{2}}$, as well as the a-priori assumption \reff{A-priAssum:eq1} for the numerical solution at the previous time steps, a rough error estimate is available:
\begin{equation}\label{propr:eq0}
\|\td\rho^{n+1} \|_2 \leq \dt^3 + h^3.
\end{equation}
\end{proposition}

\begin{proof}
Taking a discrete inner product with \reff{ErrFun:eq0-1} by $\td v^{n+\frac{1}{2}}$ leads to
\begin{equation}\label{propr:eq1}
\begin{aligned}
\langle \td{\rho}^{n+1} , \td v^{n+\frac{1}{2}} \rangle +\dt \langle \hat{\rho}^{n+\frac{1}{2}}\nabla_h \td v^{n+\frac{1}{2}} , \nabla_h\td v^{n+\frac{1}{2}} \rangle =& \langle \td{\rho}^{n} , \td v^{n+\frac{1}{2}} \rangle +\dt \langle \tau^{n+\frac{1}{2}}_\rho , \td v^{n+\frac{1}{2}} \rangle \\
&-\dt \langle \td{\hat{\rho}}^{n+\frac{1}{2}}\nabla_h \calV^{n+\frac{1}{2}} , \nabla_h\td v^{n+\frac{1}{2}} \rangle .
\end{aligned}
\end{equation}
Applying the separation estimate \reff{lemim}  for the mobility functions $\hat{\rho}^{n+\frac{1}{2}}$, we obtain the following inequality:
\begin{equation}\label{propr:eq2}
\langle \hat{\rho}^{n+\frac{1}{2}}\nabla_h \td v^{n+\frac{1}{2}} , \nabla_h\td v^{n+\frac{1}{2}} \rangle \geq \frac{\ve^*}{2}\| \nabla_h \td v^{n+\frac{1}{2}}\|^2_2.
\end{equation}
By the mean-free property for the local truncation error terms, the following estimate is obvious:
\begin{equation}\label{propr:eq3}
\langle \tau^{n+\frac{1}{2}}_\rho , \td v^{n+\frac{1}{2}} \rangle \leq \| \tau^{n+\frac{1}{2}}_\rho\|_{-1,h}\cdot \|\nabla_h\td v^{n+\frac{1}{2}} \|_2\leq \frac{2}{\ve^*}\| \tau^{n+\frac{1}{2}}_\rho\|^2_{-1,h}+\frac{\ve^*}{8}\|\nabla_h\td v^{n+\frac{1}{2}} \|_2^2.
\end{equation}
An application of the Cauchy inequality reveals that
\begin{equation}\label{propr:eq4}
\langle \td{\rho}^{n} , \td v^{n+\frac{1}{2}} \rangle \leq \|\td{\rho}^{n} \|_{-1,h}\cdot \| \nabla_h\td v^{n+\frac{1}{2}}\|_2\leq \frac{2}{\ve^*\dt}\|\td{\rho}^{n} \|^2_{-1,h}+\frac{\ve^*}{8}\dt\|\nabla_h\td v^{n+\frac{1}{2}} \|_2^2.
\end{equation}
Using discrete H$\ddot{o}$lder and Young's inequalities for the last term on the right hand side of \reff{propr:eq1}, we have
\begin{equation}\label{propr:eq5}
\begin{aligned}
-\ciptwo{\td{\hat{\rho}}^{n+\frac{1}{2}}\nabla_h \calV^{n+\frac{1}{2}}}{\nabla_h\td v^{n+\frac{1}{2}}}&\leq \|\nabla_h \calV^{n+\frac{1}{2}} \|_\infty \cdot\|\td{\hat{\rho}}^{n+\frac{1}{2}} \|_2\cdot\|\nabla_h\td v^{n+\frac{1}{2}} \|_2\\
&\leq C^* \| \td{\hat{\rho}}^{n+\frac{1}{2}}\|_2\cdot\|\nabla_h\td v^{n+\frac{1}{2}} \|_2\\
&\leq C^* \Big( \frac32 \| \td \rho^{n}\|_2+ \frac12 \|\td \rho^{n-1} \|_2 + \dt^4 \Big)\cdot\|\nabla_h\td v^{n+\frac{1}{2}} \|_2\\
&\leq \hat{C}_1 \Big( 3 \| \td{\rho}^{n}\|^2_2 +\| \td{\rho}^{n-1}\|^2_2 + \dt^8 \Big)+\frac{\ve^*}{8}\|\nabla_h\td v^{n+\frac{1}{2}} \|^2_2,
\end{aligned}
\end{equation}
where $\hat{C}_1=\frac{4 (C^*)^2}{\ve^*}$. Therefore, a substitution of (\ref{propr:eq2}-\ref{propr:eq5}) into \reff{propr:eq1} gives
\begin{equation}\label{propr:eq6}
\begin{aligned}
\langle \td{\rho}^{n+1} , \td v^{n+\frac{1}{2}} \rangle +\frac{\ve^*\dt}{8}\| \nabla_h \td v^{n+\frac{1}{2}}\|^2_2 \leq &\frac{2}{\ve^*\dt}\|\td{\rho}^{n} \|^2_{-1,h}+ \frac{2\dt}{\ve^*}\| \tau^{n+\frac{1}{2}}_\rho\|^2_{-1,h}  \\
 & +\hat{C}_1 \dt ( 3 \| \td{\rho}^{n}\|^2_2+\| \td{\rho}^{n-1}\|^2_2 + \dt^8 ) .
\end{aligned}
\end{equation}
Meanwhile, the numerical error evolutionary equation~\eqref{ErrFun:eq0-2} is equivalent to ${\cal L}_1 \tilde{\phi}^{n+1} = {\cal L}_2 \tilde{\phi}^n
   + \frac{\chi}{2} ( \tilde{\rho}^{n+1} + \tilde{\rho}^n )  + \tau_\phi^{n+\frac12}$, so that the linear error terms could be rewritten as
\[
   \tilde{\phi}^{n+1} =  {\cal L}_1^{-1} {\cal L}_2 \tilde{\phi}^n
   +  \frac{\chi}{2} {\cal L}_1^{-1}  ( \tilde{\rho}^{n+1} + \tilde{\rho}^n )
   + {\cal L}_1^{-1} \tau_\phi^{n+\frac12}.
\]
This in turn gives 
\[
   - \frac{\chi}{2}(\tilde{\phi}^{n+1} + \tilde{\phi}^n)
  +  \frac{\chi^2}{4 \theta} \dt (\tilde{\rho}^{n+1} - \tilde{\rho}^n)
  = \td \psi^n + { \cal G}_h \tilde{\rho}^{n+1},
\]
where
\[
\td \psi^n = - \frac{\chi}{2}\calL_1^{-1}\calL_2 \tilde{\phi}^{n}
   - \frac{\chi^2}{4}\calL_1^{-1} \tilde{\rho}^n - \frac{\chi}{2} \tilde{\phi}^n
    - \frac{\chi^2}{4 \theta}  \dt \tilde{\rho}^n
    - \frac{\chi}{2} {\cal L}_1^{-1} \tau_\phi^{n+\frac12}.
\]
Subsequently, the following bounds could be derived:
\begin{align}
  &
  \| \calL_1^{-1}\calL_2 \tilde{\phi}^{n}  \|_2 \le \| \tilde{\phi}^n \|_2 , \, \,
  \| \frac{\chi}{2}\calL_1^{-1}\calL_2 \tilde{\phi}^{n}
   + \frac{\chi}{2} \tilde{\phi}^n \|_2 \le \chi \| \tilde{\phi}^n \|_2 ,  \, \,
   \| {\cal L}_1^{-1} \tau_\phi^{n+\frac12}  \|_2
   \le \theta^{-1} \dt \| \tau_\phi^{n+\frac12}  \|_2 , \nonumber 
\\
  &
  \| \calL_1^{-1} \tilde{\rho}^n \|_2 \le \theta^{-1} \dt \| \tilde{\rho}^n \|_2 , \quad
  \| \frac{\chi^2}{4}\calL_1^{-1} \tilde{\rho}^n  +  \frac{\chi^2}{4 \theta}  \dt \tilde{\rho}^n \|_2
  \le  \frac{\chi^2}{2 \theta}  \dt \| \tilde{\rho}^n \|_2  , \nonumber 
\\
  &
  \mbox{so that} \quad \| \td \psi^n \|_2 \le \chi \| \tilde{\phi}^n \|_2
  + \frac{\chi^2}{2 \theta}  \dt \| \tilde{\rho}^n \|_2 + \frac{\chi}{2 \theta}  \dt  \| \tau_\phi^{n+\frac12}  \|_2
  \le C ( \dt^\frac{15}{4} + h^\frac{15}{4} ) , 
   \label{propr:eq8-2}
\end{align}
in which the inequalities $\| {\cal L}_1^{-1} {\cal L}_2  f \|_2 \le \| f \|_2$, $\| {\cal L}_1^{-1} f \|_2 \le \theta^{-1} \dt \| f \|_2$, and the a-priori assumption \reff{A-priAssum:eq1}, have been repeatedly applied in the derivation. Of course, an application of inverse inequality indicates that
\begin{equation}
  \| \td \psi^n \|_\infty \le \frac{C \| \td \psi^n \|_2}{h^\frac32}
  \le \frac{C ( \dt^\frac{15}{4} + h^\frac{15}{4} )}{h^\frac32}
  \le C ( \dt^\frac{9}{4} + h^\frac{9}{4} ) \le h ,  \quad
  \mbox{since $\lambda_1 h \le \dt \le \lambda_2 h$} .
  \label{propr:eq8-3}
\end{equation}
As a consequence, an application of the rough bound control \reff{Lemr:eq1} (in Lemma \ref{Lem:CA1}) gives
\begin{equation}\label{propr:eq8-4}
\begin{aligned}
\langle \td{\rho}^{n+1} , \gamma \td S^{n+\frac{1}{2}} + \td \psi^n \rangle
\geq \frac{C^*}{6}\gamma  L^* h^3- \ckC_2 ( \gamma^2 \|\td \rho^{n} \|^2_2 + \| \td \psi^n \|_2^2 ) .
\end{aligned}
\end{equation}
Moreover, the monotonicity estimate~\eqref{lem:mono-0} of the operator $\calG_h$ (in Lemma~\ref{lem:mono}) implies that
\begin{equation}\label{propr:eq8-5}
\begin{aligned}
\ciptwo{\td{\rho}^{n+1}}{\calG_h\td \rho^{n+1}} \ge 0 .
\end{aligned}
\end{equation}
In turn,  a combination of~\eqref{propr:eq8-4} and \eqref{propr:eq8-5} leads to
\begin{equation}
  \langle \td{\rho}^{n+1} , \td v^{n+\frac{1}{2}} \rangle
  \ge \frac{C^*}{6}\gamma  L^* h^3 - \ckC_2 ( \gamma^2\|\td \rho^{n} \|^2_2 + \| \td \psi^n \|_2^2 ) .
    \label{propr:eq8-7}
\end{equation}
Its combination with \eqref{propr:eq6} reveals that
\begin{equation}
\begin{aligned}
  \frac{C^*}{6} \gamma L^* h^3 + \frac{\ve^*\dt}{8}\| \nabla_h \td v^{n+\frac{1}{2}}\|^2_2
  \le &\frac{2}{\ve^*\dt}\|\td{\rho}^{n} \|^2_{-1,h}
  + \frac{2\dt}{\ve^*}\| \tau^{n+\frac{1}{2}}_\rho\|^2_{-1,h}
  + \ckC_2 \| \td \psi^n \|_2^2
\\
  &
  + ( \ckC_2\gamma^2 + 1) \|\td \rho^{n} \|^2_2
  +\hat{C}_1 \dt ( \| \td{\rho}^{n-1}\|^2_2 + \dt^8 ) ,  \, \,  \mbox{if} \, \, \, 3\hat{C}_1 \dt \le 1 .
\end{aligned}
  \label{propr:eq9}
\end{equation}
The following bounds for the right hand side terms are available, based on the a-priori assumption \reff{A-priAssum:eq1}, the preliminary estimate~\eqref{propr:eq8-2}, as well as the higher order truncation error accuracy:
\begin{equation}\label{propr:eq10}
\begin{aligned}
&\frac{2}{\ve^*\dt}\|\td{\rho}^{n} \|^2_{-1,h}\leq \frac{2C}{\ve^*\dt}\|\td{\rho}^{n} \|^2_2\leq C (\dt^{\frac{13}{2}}+h^{\frac{13}{2}}), \\
& \ckC_2 \| \td \psi^n \|_2^2 , \, \, ( \ckC_2 \gamma^2+ 1) \|\td \rho^{n} \|^2_2
\le  C(\dt^{\frac{15}{2}}+h^{\frac{15}{2}}), \\
&\frac{2\dt}{\ve^*}\|\tau^{n+\frac{1}{2}}_\rho\|^2_{-1,h}\leq C\dt\|\tau^{n+\frac{1}{2}}_\rho\|^2_2\leq C(\dt^9+\dt h^8), \\
& \hat{C}_1 \dt \| \td{\rho}^{n-1}\|^2_2  \le C \dt (\dt^{\frac{15}{2}}+h^{\frac{15}{2}})
\le C ( \dt^{\frac{17}{2}}+ \dt h^{\frac{15}{2}} ). 
\end{aligned}
\end{equation}
 Again, the inequality $\| f\|_{-1,h}\leq C\| f\|_2$ and the linear refinement requirement $\lambda_1 h\leq \dt\leq\lambda_2 h$, have been used. Going back to~\eqref{propr:eq9}, we arrive at
\[
 \frac{C^*}{6} \gamma L^* h^3\leq C(\dt^{\frac{13}{2}}+h^{\frac{13}{2}}).
\]
If $L^*\geq 1$, this inequality could make a contradiction, provided that $\dt$ and $h$ are sufficiently small. Therefore, we conclude that $L^*=0$. In turn, an improved estimate~\eqref{Lemr:eq2}, as given by Lemma~\ref{Lem:CA1}, becomes available. As a direct consequence, we obtain
\begin{equation}
\begin{aligned}
  \langle \td{\rho}^{n+1} , \td v^{n+\frac{1}{2}} \rangle
  \ge & \ckC_3 \| \td \rho^{n+1} \|_2^2 - \ckC_2 ( \gamma^2\|\td \rho^{n} \|^2_2 + \| \td \psi^n \|_2^2 ) ,
  \quad \mbox{so that}
\\
  \ckC_3 \| \td \rho^{n+1} \|_2^2 + \frac{\ve^*\dt}{8}\| \nabla_h \td v^{n+\frac{1}{2}}\|^2_2
  \le &\frac{2}{\ve^*\dt}\|\td{\rho}^{n} \|^2_{-1,h}
  + \frac{2\dt}{\ve^*}\| \tau^{n+\frac{1}{2}}_\rho\|^2_{-1,h}
  + \ckC_2 \| \td \psi^n \|_2^2 + \hat{C}_1 \dt^9
\\
  &
  + ( \ckC_2 \gamma^2+ 1) \|\td \rho^{n} \|^2_2
  +\hat{C}_1 \dt \| \td{\rho}^{n-1}\|^2_2  \le C(\dt^{\frac{13}{2}}+h^{\frac{13}{2}}) .
\end{aligned}
    \label{propr:eq11}
\end{equation}
In particular, we see that
\begin{equation}
  \| \td \rho^{n+1}\|_2 \le  C (\dt^{\frac{13}{4}} + h^{\frac{13}{4}})
  \le \dt^3 + h^3 ,   \label{propr:eq12}
\end{equation}
under the linear refinement requirement $\lambda_1 h\leq \dt\leq\lambda_2 h$. This inequality is exactly the rough error estimate \reff{propr:eq0}, and the proof of Proposition~\ref{prop:rough} is completed.
\qed
\end{proof}

With the rough error estimate~\eqref{propr:eq0} at hand, we are able to establish the $W_h^{1,\infty}$ bound of the numerical solution for the density variable. A direct application of 3-D inverse inequality gives
\begin{equation}\label{propr:eq13}
\begin{aligned}
  &
 \|\td\rho^{n+1} \|_\infty \le \frac{C\|\td\rho^{n+1} \|_2}{h^{\frac{3}{2}}} \leq C(\dt^\frac32 +h^\frac32 )
 \leq \frac{\ve^*}{2} ,
\\
  &
  \| \nabla_h \td \rho^{n+1} \|_\infty \le \frac{2 \|\td\rho^{n+1} \|_\infty}{h} \leq C(\dt^\frac12 +h^\frac12 )
 \le 1 ,
\end{aligned}
\end{equation}
under the same linear refinement requirement. In turn, the following separation is valid at time step $t_{n+1}$:
\begin{equation}\label{propr:eq14}
\frac{\ve^*}{2}\leq\rho^{n+1} \leq C^*+\frac{\ve^*}{2} = \ckC_0 .
\end{equation}
This $\| \cdot\|_\infty$ bound will play a very important role in the refined error estimate. Moreover, a maximum norm bound also becomes available for $\nabla_h \rho^{n+1}$:
\begin{equation}\label{propr:eq15}
\|\nabla_h\rho^{n+1} \|_\infty \leq\|\nabla_h\crho^{n+1} \|_\infty +\|\nabla_h\td\rho^{n+1} \|_\infty
 \leq C^*+1 =\tdC_0.
\end{equation}
Meanwhile, the following bound, in terms of the discrete temporal derivative for the numerical solution at time step $t_{n+1}$, will be used in the refined error estimate:
\begin{equation}
\begin{aligned}
  &
  \| \tilde{\rho}^{n+1} - \tilde{\rho}^n \|_\infty \le  \| \tilde{\rho}^{n+1} \|_\infty + \| \tilde{\rho}^n \|_\infty
  \le 2 C (\dt^\frac32 + h^\frac32)  \le \dt ,
\\
  &
  \| \rho^{n+1} - \rho^n \|_\infty \le \| \crho^{n+1} - \crho^n \|_\infty
  + \| \tilde{\rho}^{n+1} - \tilde{\rho}^n \|_\infty  \le ( C^* + 1 ) \dt = \tilde{C}_0 \dt .
\end{aligned}
  \label{propr:eq16}
\end{equation}
In particular, the following observation is made
\begin{equation}
\begin{aligned}
  \| \rho^{n+1} - \hat{\rho}^{n+\frac12} \|_\infty 
  \le & \| (  \rho^{n+1} - \rho^n ) - \frac12 ( \rho^n - \rho^{n-1} ) \|_\infty + C \dt^4
\\
  \le &
  \|  \rho^{n+1} - \rho^n \|_\infty + \frac12 \| \rho^n - \rho^{n-1} \|_\infty + C \dt^4
  \le \frac32 \tdC_0 \dt .
\end{aligned}
  \label{propr:eq17-1}
\end{equation}
Meanwhile, by the fact that $\frac{\ve^*}{2} \le \rho^{n+1} , \,  \hat{\rho}^{n+\frac12} \le \tilde{C}_0$, it is clear that
\begin{equation}
  \frac34 \le \frac{\hat{\rho}^{n+\frac12}}{\rho^{n+1}} \le \frac54 ,  \quad
  \mbox{at a point-wise level, provided that $\dt$ is sufficiently small.}
   \label{propr:eq17-2}
\end{equation}

\subsection{A refined error estimate}
The rough error estimate~\eqref{propr:eq0} is not able to go through an induction argument. Therefore, a refined error estimate is needed to accomplish a closed loop of convergence analysis. Because of the Crank-Nicolson-style temporal discretization in the numerical design, the following preliminary estimate is necessary to control the nonlinear errors associated with the logarithmic diffusion part. 
The technical details of the proof are provided in Appendix \ref{Ap:B}.

\begin{proposition}\label{prop:CA2}
Assume that the a-priori $\| \cdot\|_\infty$ estimate \reff{presolns:eq2}-\reff{lemim} and the rough $\| \cdot\|_\infty$ estimates \reff{propr:eq14}-\reff{propr:eq17-2} hold for the numerical solution at the previous and next time steps, respectively. There exist positive constants $\tdC_1$, $M_1$, dependent only on $\ve^*$, $C^*$, $\gamma$, $\tdC_0$ and $|\Omega|$ such that
\begin{equation}\label{Lemn:eq0}
\begin{aligned}
\gamma \ciptwo{\hat{\rho}^{n+\frac{1}{2}}\nabla_h\td S^{n+\frac{1}{2}} }{\nabla_h ( \td\rho^{n+1} + \td \rho^n ) }\geq & \frac{\gamma }{4} \|\nabla_h ( \td\rho^{n+1} + \td \rho^n ) \|^2_2
  - \tdC_1 (\|\td \rho^{n+1} \|^2_2+\|\td \rho^{n} \|^2_2 ) - M_1 h^8  .
\end{aligned}
\end{equation}
\end{proposition}

Now we look at the refined error estimate. Taking a discrete inner product with \reff{ErrFun:eq0-1} by $(\td\rho^{n+1} + \td \rho^n)$ gives
\begin{equation}\label{Lemd:eq1}
\begin{aligned}
&\frac{1}{\dt} (\|\td\rho^{n+1} \|^2_2-\|\td\rho^{n} \|^2_2 )+ \ciptwo{\hat{\rho}^{n+\frac{1}{2}}\nabla_h \td v^{n+\frac{1}{2}} }{\nabla_h ( \td\rho^{n+1} + \td \rho^n ) } \\
&\qquad=- \ciptwo{\td{\hat{\rho}}^{n+\frac{1}{2}}\nabla_h \calV^{n+\frac{1}{2}}}{\nabla_h ( \td\rho^{n+1} + \td \rho^n ) } + \langle \tau^{n+\frac{1}{2}}_\rho , \td\rho^{n+1} + \td \rho^n \rangle .
\end{aligned}
\end{equation}
The first term on the right hand side could be analyzed in a standard way:
\begin{equation}\label{Lemd:eq2}
\begin{aligned}
- \ciptwo{\td{\hat{\rho}}^{n+\frac{1}{2}}\nabla_h \calV^{n+\frac{1}{2}}}{\nabla_h ( \td\rho^{n+1} + \td \rho^n) } \le & \|\nabla_h \calV^{n+\frac{1}{2}} \|_\infty\cdot\|\td{\hat{\rho}}^{n+\frac{1}{2}} \|_2 \cdot \|\nabla_h ( \td\rho^{n+1} + \td \rho^n ) \|_2 \\
\le &  C^*( \frac32 \|\td{\rho}^n \|_2 + \frac12 \|\td{\rho}^{n-1} \|_2 + \dt^4)
 \|\nabla_h ( \td\rho^{n+1} + \td \rho^n ) \|_2 \\
\le &  \frac{( C^*)^2}{\gamma} ( 24 \|\td{\rho}^n \|_2^2 + 8\gamma  \|\td{\rho}^{n-1} \|_2^2 + 8 \dt^8 )
\\
  &
 + \frac{\gamma}{16} \|\nabla_h ( \td\rho^{n+1} + \td \rho^n ) \|_2^2.
\end{aligned}
\end{equation}
The Cauchy inequality is applied to bound the local truncation error term:
\begin{equation}\label{Lemd:eq3}
  \langle \tau^{n+\frac{1}{2}}_\rho , \td \rho^{n+1} + \td \rho^n \rangle
  \le \| \tau^{n+\frac{1}{2}}_\rho \| \cdot \| \td \rho^{n+1} + \td \rho^n \|_2
  \le \frac12 \|\tau^{n+\frac{1}{2}}_\rho \|^2_2
  + ( \|\td\rho^{n+1} \|^2_2 + \|\td\rho^n \|^2_2 ) .
\end{equation}
For the nonlinear term on the left hand side, we separate it into three parts:
\begin{equation}\label{Lemd:eq4}
\begin{aligned}
&\ciptwo{\hat{\rho}^{n+\frac{1}{2}}\nabla_h \td v^{n+\frac{1}{2}} }{\nabla_h ( \td\rho^{n+1} + \td \rho^n ) }
\\
&\quad=\gamma \ciptwo{\hat{\rho}^{n+\frac{1}{2}} \nabla_h \td S^{n+\frac{1}{2}} }{\nabla_h ( \td\rho^{n+1} + \td \rho^n ) } - \frac{\chi}{2} \ciptwo{\hat{\rho}^{n+\frac{1}{2}} \nabla_h ( \td \phi^{n+1} + \td \phi^n ) }{\nabla_h ( \td \rho^{n+1} + \td \rho^n ) } \\
&\qquad+ \frac{\chi^2 \dt}{4 \theta} \ciptwo{\hat{\rho}^{n+\frac{1}{2}} \nabla_h ( \td \rho^{n+1} - \td \rho^n ) }{\nabla_h ( \td \rho^{n+1} + \td \rho^n ) } .
\end{aligned}
\end{equation}
The second part could be controlled by a direct application of the Cauchy inequality:
\begin{equation}
\begin{aligned}
  &
  \frac{\chi}{2} \ciptwo{\hat{\rho}^{n+\frac{1}{2}} \nabla_h ( \td \phi^{n+1} + \td \phi^n ) }{\nabla_h ( \td \rho^{n+1} + \td \rho^n ) }
\\
  &\qquad \le \frac{\chi}{2} \| \hat{\rho}^{n+\frac{1}{2}} \|_\infty
  \cdot \| \nabla_h ( \td \phi^{n+1} + \td \phi^n ) \|_2
  \cdot \| \nabla_h ( \td \rho^{n+1} + \td \rho^n ) \|_2
\\
  &\qquad \le  \frac{\chi \ckC_0}{2}
  \cdot \| \nabla_h ( \td \phi^{n+1} + \td \phi^n ) \|_2
  \cdot \| \nabla_h ( \td \rho^{n+1} + \td \rho^n ) \|_2
\\
  &\qquad \le  \frac{2 \chi^2 \ckC_0^2}{\gamma} ( \| \nabla_h \td \phi^{n+1} \|_2^2 + \| \nabla_h \td \phi^n \|_2^2 )
  + \frac{\gamma}{16} \| \nabla_h ( \td \rho^{n+1} + \td \rho^n ) \|_2^2 .
\end{aligned}
  \label{Lemd:eq5}
\end{equation}
In terms of the third part on the right hand side of~\eqref{Lemd:eq4}, we begin with a point-wise vector identity: $\nabla_h ( \td \rho^{n+1} - \td \rho^n ) \cdot \nabla_h ( \td \rho^{n+1} + \td \rho^n )
= | \nabla_h \td \rho^{n+1} |^2 - | \nabla_h \td \rho^n |^2$. This in turn leads to
\begin{equation}
\begin{aligned}
  &
  \ciptwo{\hat{\rho}^{n+\frac{1}{2}} \nabla_h ( \td \rho^{n+1} - \td \rho^n ) }{\nabla_h ( \td \rho^{n+1} + \td \rho^n ) }
\\
&\qquad=
  \ciptwo{\hat{\rho}^{n+\frac{1}{2}} }{| \nabla_h \td \rho^{n+1} |^2 - | \nabla_h \td \rho^n |^2}
\\
&\qquad =
    \ciptwo{\hat{\rho}^{n+\frac{1}{2}} }{| \nabla_h \td \rho^{n+1} |^2}
  - \ciptwo{\hat{\rho}^{n-\frac12} }{ | \nabla_h \td \rho^n |^2}
  - \ciptwo{\hat{\rho}^{n+\frac12} - \hat{\rho}^{n-\frac12} }{ | \nabla_h \td \rho^n |^2} .
\end{aligned}
  \label{Lemd:eq6-1}
\end{equation}
Moreover, for the last term in the rewritten expansion, an application of the preliminary estimate~\eqref{lemim} implies that
\begin{equation}
\begin{aligned}
  &
  \ciptwo{\hat{\rho}^{n+\frac12} - \hat{\rho}^{n-\frac12} }{ | \nabla_h \td \rho^n |^2}
  \le  \| \hat{\rho}^{n+\frac12} - \hat{\rho}^{n-\frac12} \|_\infty \cdot \| \nabla_h \td \rho^n \|_2^2
  \le 2 \td C_0 \dt \| \nabla_h \td \rho^n \|_2^2 .
\end{aligned}
  \label{Lemd:eq6-2}
\end{equation}
Subsequently, a combination of~\eqref{Lemd:eq6-1} and \eqref{Lemd:eq6-2} yields
\begin{equation}
\begin{aligned}
  &
  \frac{\chi^2 \dt}{4 \theta} \ciptwo{\hat{\rho}^{n+\frac{1}{2}} \nabla_h ( \td \rho^{n+1} - \td \rho^n ) }{\nabla_h ( \td \rho^{n+1} + \td \rho^n ) }
\\
 & \qquad   \ge  
    \frac14 \chi^2 \theta^{-1} \dt \Big( \ciptwo{\hat{\rho}^{n+\frac{1}{2}} }{| \nabla_h \td \rho^{n+1} |^2}
  - \ciptwo{\hat{\rho}^{n-\frac12} }{ | \nabla_h \td \rho^n |^2} \Big)
  - \frac12 \chi^2 \td C_0 \theta^{-1} \dt^2 \| \nabla_h \td \rho^n \|_2^2 .
\end{aligned}
  \label{Lemd:eq6-3}
\end{equation}
Meanwhile, the nonlinear diffusion error estimate~\eqref{Lemn:eq0} is valid, as stated in Proposition~\ref{prop:CA2}. A substitution of \eqref{Lemn:eq0}, \eqref{Lemd:eq5}, and \eqref{Lemd:eq6-3} into \eqref{Lemd:eq4} results in
\begin{equation}
\begin{aligned}
&\ciptwo{\hat{\rho}^{n+\frac{1}{2}}\nabla_h \td v^{n+\frac{1}{2}} }{\nabla_h ( \td\rho^{n+1} + \td \rho^n ) }
\\
 & \quad \ge  
    \frac14 \chi^2 \theta^{-1} \dt \Big( \ciptwo{\hat{\rho}^{n+\frac{1}{2}} }{| \nabla_h \td \rho^{n+1} |^2}
    - \ciptwo{\hat{\rho}^{n-\frac12} }{ | \nabla_h \td \rho^n |^2} \Big)
    + \frac{3\gamma }{16} \| \nabla_h ( \td \rho^{n+1} + \td \rho^n ) \|_2^2 - M_1 h^8
\\
& \qquad - \tdC_1 (\|\td \rho^{n+1} \|^2_2+\|\td \rho^{n} \|^2_2 )
    - \frac{2 \chi^2 \ckC_0^2}{\gamma } ( \| \nabla_h \td \phi^{n+1} \|_2^2 + \| \nabla_h \td \phi^n \|_2^2 )
    - \frac12 \chi^2 \td C_0 \theta^{-1} \dt^2 \| \nabla_h \td \rho^n \|_2^2 .
\end{aligned}
  \label{Lemd:eq7}
\end{equation}
Its combination with~\eqref{Lemd:eq1}-\eqref{Lemd:eq3} reveals that
\begin{equation}
\begin{aligned}
&\frac{1}{\dt} (\|\td\rho^{n+1} \|^2_2-\|\td\rho^{n} \|^2_2 )
  + \frac14 \chi^2 \theta^{-1} \dt ( \langle \hat{\rho}^{n+\frac{1}{2}} , | \nabla_h \td \rho^{n+1} |^2 \rangle
    - \langle \hat{\rho}^{n-\frac12} , | \nabla_h \td \rho^n |^2 \rangle ) + \frac{\gamma }{8} \| \nabla_h ( \td \rho^{n+1} + \td \rho^n ) \|_2^2
\\
  &~\le
     ( \tdC_1 + 1 ) \|\td \rho^{n+1} \|^2_2 + ( \tdC_1 + \frac12 \chi^2 \td C_0 \theta^{-1} \hat{C}_3^2
     + 24 \frac{(C^*)^2}{\gamma } + 1 ) \|\td \rho^{n} \|^2_2
     + 8\frac{(C^*)^2}{\gamma } \|\td \rho^{n-1} \|^2_2
\\
  &\qquad 
    + \frac{2 \chi^2 \ckC_0^2}{\gamma } ( \| \nabla_h \td \phi^{n+1} \|_2^2 + \| \nabla_h \td \phi^n \|_2^2 )
   + \frac12 \|\tau^{n+\frac{1}{2}}_\rho \|^2_2  + \frac{8 (C^*)^2}{\gamma } \dt^8 +  M_1 h^8,
\end{aligned}
  \label{Lemd:eq8}
\end{equation}
in which an inverse inequality and $\dt \| \nabla_h \td \rho^n \|_2 \le \hat{C}_3 \| \td \rho^n \|_2$ (under the linear refinement requirement  $\lambda_1 h \le \dt \le \lambda_2 h$) has been used.

The analysis for the numerical error evolutionary equation~\eqref{ErrFun:eq0-2} takes a much simpler form, because of its linear nature. Taking a discrete inner product with \reff{ErrFun:eq0-2} by $(\td\phi^{n+1} + \td \phi^n)$ gives
\begin{equation}\label{Lemd:eq9}
\begin{aligned}
&\frac{\theta}{\dt} (\|\td \phi^{n+1} \|^2_2-\|\td \phi^{n} \|^2_2 ) + \frac{\mu}{2}  \| \nabla_h ( \td \phi^{n+1} + \td \phi^n ) \|_2^2 + \frac{\alpha}{2}  \| \td \phi^{n+1} + \td \phi^n \|_2^2
\\
&\quad =  \frac{\chi}{2}  \langle \td \rho^{n+1} + \td \rho^n , \td \phi^{n+1} + \td \phi^n \rangle
 + \langle \tau^{n+\frac{1}{2}}_\phi , \td \phi^{n+1} + \td \phi^n \rangle
\\
  & \quad   \le 
    \frac{\chi}{2} ( \| \td \rho^{n+1} \|_2^2 + \| \td \rho^n \|_2^2 )
  + (  \frac{\chi}{2} + 1) ( \| \td \phi^{n+1} \|_2^2 + \| \td \phi^n \|_2^2 )
  + \frac12 \|\tau^{n+\frac{1}{2}}_\phi \|^2_2 ,
\end{aligned}
\end{equation}
in which the Cauchy inequality has been repeatedly applied. Meanwhile, we need a further $H^1$ error estimate for the density variable, to balance the terms $\| \nabla_h \td \phi^{n+1} \|_2^2$ and $\| \nabla_h \td \phi^n \|_2^2$ on the right hand side of~\eqref{Lemd:eq8}. Taking a discrete inner product with \reff{ErrFun:eq0-2} by $- \Delta_h (\td\phi^{n+1} + \td \phi^n)$ indicates
\begin{equation}
\begin{aligned}
&\frac{\theta}{\dt} (\| \nabla_h \td \phi^{n+1} \|^2_2 - \| \nabla_h \td \phi^{n} \|^2_2 ) + \frac{\mu}{2}  \| \Delta_h ( \td \phi^{n+1} + \td \phi^n ) \|_2^2
 + \frac{\alpha}{2}  \| \nabla_h ( \td \phi^{n+1} + \td \phi^n ) \|_2^2
\\
& \quad =  - \frac{\chi}{2}  \langle \td \rho^{n+1} + \td \rho^n , \Delta_h ( \td \phi^{n+1} + \td \phi^n ) \rangle
 - \langle \tau^{n+\frac{1}{2}}_\phi , \Delta_h ( \td \phi^{n+1} + \td \phi^n ) \rangle
\\
& \quad \le 
  \chi^2 \mu^{-1} ( \| \td \rho^{n+1} \|_2^2 + \| \td \rho^n \|_2^2 )
  + 2 \mu^{-1} \|\tau^{n+\frac{1}{2}}_\phi \|^2_2
  +  \frac{\mu}{4}  \| \Delta_h  ( \td \phi^{n+1} + \td \phi^n ) \|_2^2  .
\end{aligned}
  \label{Lemd:eq10-1}
\end{equation}
This is equivalent to
\begin{equation}
\begin{aligned}
&\frac{\theta}{\dt} (\| \nabla_h \td \phi^{n+1} \|^2_2 - \| \nabla_h \td \phi^{n} \|^2_2 )
+ \frac{\mu}{4}  \| \Delta_h ( \td \phi^{n+1} + \td \phi^n ) \|_2^2
 + \frac{\alpha}{2}  \| \nabla_h ( \td \phi^{n+1} + \td \phi^n ) \|_2^2
\\
  & \qquad \le  
  \chi^2 \mu^{-1} ( \| \td \rho^{n+1} \|_2^2 + \| \td \rho^n \|_2^2 )
  + 2 \mu^{-1} \|\tau^{n+\frac{1}{2}}_\phi \|^2_2   .
\end{aligned}
  \label{Lemd:eq10-2}
\end{equation}
Therefore, a combination of~\eqref{Lemd:eq8}, \eqref{Lemd:eq9} and \eqref{Lemd:eq10-2} leads to
\begin{equation}
\begin{aligned}
&\frac{1}{\dt} \Big( \|\td\rho^{n+1} \|^2_2 - \|\td\rho^{n} \|^2_2
+ \theta ( \|\td \phi^{n+1} \|^2_2 - \|\td \phi^n \|^2_2
+ \| \nabla_h \td \phi^{n+1} \|^2_2 - \| \nabla_h \td \phi^n \|^2_2 )  \Big)
\\
  &
  \quad+ \frac14 \chi^2 \theta^{-1} \dt ( \langle \hat{\rho}^{n+\frac{1}{2}} , | \nabla_h \td \rho^{n+1} |^2 \rangle
    - \langle \hat{\rho}^{n-\frac12} , | \nabla_h \td \rho^n |^2 \rangle )
\\
  &
    \quad+ \frac{\gamma }{8} \| \nabla_h ( \td \rho^{n+1} + \td \rho^n ) \|_2^2
    + \frac{\mu}{4}  \| \Delta_h ( \td \phi^{n+1} + \td \phi^n ) \|_2^2
\\
 & \le 
     ( \tdC_1 + 1 + \frac{\chi}{2} + \chi^2 \mu^{-1} ) \|\td \rho^{n+1} \|^2_2
     + ( \tdC_1 + \frac12 \chi^2 \td C_0 \theta^{-1} \hat{C}_3^2
     + 24 \frac{(C^*)^2}{\gamma } + 1 + \frac{\chi}{2} + \chi^2 \mu^{-1} ) \|\td \rho^{n} \|^2_2
\\
  &
     \quad+ 8 \frac{(C^*)^2}{\gamma } \|\td \rho^{n-1} \|^2_2
     + (  \frac{\chi}{2} + 1) ( \| \td \phi^{n+1} \|_2^2 + \| \td \phi^n \|_2^2 )
    + \frac{2 \chi^2 \ckC_0^2}{\gamma } ( \| \nabla_h \td \phi^{n+1} \|_2^2 + \| \nabla_h \td \phi^n \|_2^2 )
\\
  &
 \quad+ \frac12 \|\tau^{n+\frac{1}{2}}_\phi \|^2_2 
   + \frac12 \|\tau^{n+\frac{1}{2}}_\rho \|^2_2
   + 2 \mu^{-1} \|\tau^{n+\frac{1}{2}}_\phi \|^2_2  + 8 \frac{(C^*)^2}{\gamma } \dt^8 + M_1 h^8  .
\end{aligned}
  \label{Lemd:eq11-1}
\end{equation}
In turn, the following quantity is introduced:
\begin{equation}
  {\cal F}^{n+1} := \|\td\rho^{n+1} \|^2_2
+ \theta ( \|\td \phi^{n+1} \|^2_2  + \| \nabla_h \td \phi^{n+1} \|^2_2  )
  + \frac14 \chi^2 \theta^{-1} \dt^2
  \langle \hat{\rho}^{n+\frac{1}{2}} , | \nabla_h \td \rho^{n+1} |^2 \rangle  .
 \label{Lemd:eq11-2}
\end{equation}
Then we obtain
\begin{equation}
\begin{aligned}
&\frac{1}{\dt} ( {\cal F}^{n+1} - {\cal F}^n )
    + \frac{\gamma }{8} \| \nabla_h ( \td \rho^{n+1} + \td \rho^n ) \|_2^2
    + \frac{\mu}{4}  \| \Delta_h ( \td \phi^{n+1} + \td \phi^n ) \|_2^2
\\
  & \quad \le 
  \td C_2 ( {\cal F}^{n+1} + {\cal F}^n + {\cal F}^{n-1} )
   + \frac12 \|\tau^{n+\frac{1}{2}}_\rho \|^2_2
   + 2 \mu^{-1} \|\tau^{n+\frac{1}{2}}_\phi \|^2_2 +\frac12 \|\tau^{n+\frac{1}{2}}_\phi \|^2_2 
\\
  &
   \qquad + 8 \frac{(C^*)^2}{\gamma } \dt^8 + M_1 h^8,
\end{aligned}
\label{Lemd:eq11-3}
\end{equation}
where
\[
  \td C_2 =
  \max \Big( \tdC_1 + \frac12 \chi^2 \td C_0 \theta^{-1} \hat{C}_3^2
     + 24 \frac{(C^*)^2}{\gamma } + 1 + \frac{\chi}{2} + \chi^2 \mu^{-1}  ,
     (  \frac{\chi}{2} + 1) \theta^{-1} ,  \frac{ 2\chi^2 \ckC_0^2 \theta^{-1}}{\gamma } \Big).
\]
Consequently, an application of the discrete Gronwall inequality gives the desired higher order convergence estimate:
\begin{equation}\label{Lemd:eq12}
\begin{aligned}
  &
  {\cal F}^{n+1} \le C (\dt^8 + h^8) , \quad
  \| \td \rho^{n+1} \|_2 +  \| \td \phi^{n+1} \|_2 + \| \nabla_h \td \phi^{n+1} \|_2
  \le C ( {\cal F}^{n+1} )^\frac12 \le C (\dt^4+h^4) ,
\\
  &
  \dt \sum_{k=0}^n \| \nabla_h ( \td \rho^{k+1} + \td \rho^k ) \|_2^2
    + \dt \sum_{k=0}^n  \| \Delta_h ( \td \phi^{k+1} + \td \phi^k ) \|_2^2   \le C ( \dt^8 + h^8 ) ,
\end{aligned}
\end{equation}
in which the fourth order truncation error accuracy $\|\tau^{n+\frac{1}{2}}_\rho \|_2 , \,
\|\tau^{n+\frac{1}{2}}_\phi \|_2 \leq C(\dt^4+h^4)$ has been applied. This finishes the refined error estimate.

\subsection{Recovery of the a-priori assumption \reff{A-priAssum:eq1}}

With the help of the higher order error estimate, we see that the a-priori assumption \reff{A-priAssum:eq1} is satisfied at $t_{n+1}$:
\begin{equation}
 \|\td\rho^{n+1} \|_2 , \, \|\td\phi^{n+1} \|_2
 \leq C(\dt^4+h^4)\leq \dt^{\frac{15}{4}}+h^{\frac{15}{4}},
\end{equation}
provided that $\dt$ and $h$ are sufficiently small. Therefore, an induction analysis could be effectively applied and  the higher order convergence analysis is complete. Subsequently, a combination of \reff{Lemd:eq12} with \reff{Consis:eq0} leads to the convergence estimate \reff{th:eq1}. The proof of Theorem \ref{Th:convergence} is completed.

\section{Numerical results}
\subsection{Accuracy test}
We now test numerical accuracy of the proposed scheme (\ref{PKS2nd_1}-\ref{PKS2nd_2}) in solving the PKS system 
\begin{equation}\label{eq:ex}
\left\{
\begin{aligned}
\partial_t \rho&=\Delta \rho-\nabla\cdot(\rho\nabla\phi)+f_1,\\
\theta\partial_t\phi&=\Delta\phi-\phi+\rho+f_2,
\end{aligned}
\right.
\end{equation}
on a two-dimensional computational domain $\Omega=(0,1)^2$. The source terms $f_1$ and $f_2$ are determined by the following exact solution
\begin{equation}\label{ExS}
\left\{
\begin{aligned}
&\rho_e (x,y,t) =0.1 {\rm e}^{-t}\cos(\pi x)\cos(\pi y)+0.2,\\
&\phi_e (x,y,t) =0.1 {\rm e}^{-t}\cos(\pi x)\cos(\pi y)+0.2.\\
\end{aligned}
\right.
\end{equation}
The initial conditions are obtained by evaluating the exact solution at $t=0$. We consider homogeneous Neumann boundary condition \reff{BCs} for both $\rho$ and $\phi$.
\begin{figure}[h]
\centering
\includegraphics[scale=.55]{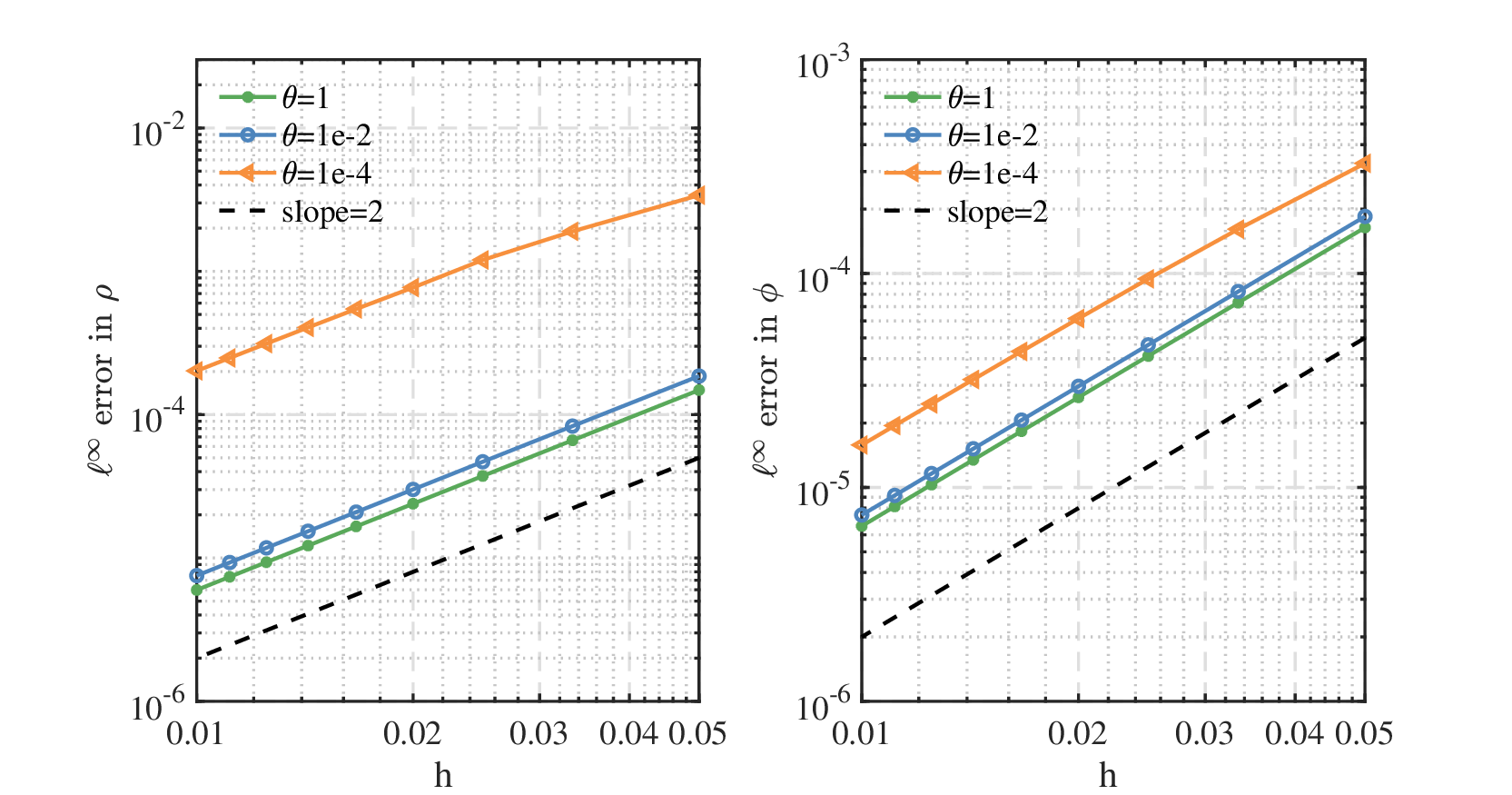}
\caption{Numerical errors (in $\ell^\infty$) of $\rho$ and $\phi$ computed by the second-order accurate scheme (\ref{PKS2nd_1}-\ref{PKS2nd_2}) at a final time $T=0.1$, with a mesh  ratio $\Delta t=h/10$. Various values of $\theta$ are considered: $\theta=1$, $\theta=1e-2$, and $\theta=1e-4$.}
\label{f:Order}
\end{figure}
We first test  numerical accuracy of the proposed scheme utilizing various spatial step size $h$ with a fixed mesh ratio $\Delta t=h/10$. Figure~\ref{f:Order} displays the $\ell^\infty$ errors and convergence orders for the density of living organisms and chemical signals at a final time $T=0.1$.  We observe that the numerical error decreases as the mesh refines and that  second order convergence rates are clearly observed for both $\rho$ and $\phi$. This verifies the second-order accuracy for the proposed numerical scheme~(\ref{PKS2nd_1}-\ref{PKS2nd_2}), in both temporal and spatial discretization. Notice that the mesh ratio is chosen for the sake of numerical accuracy test, not for stability or positivity.

\subsection{Performance test}
\subsubsection{Symmetric initial data on a square}\label{ss:sym}
In this case, we demonstrate the performance of the proposed scheme in preserving mass conservation, energy dissipation, and solution positivity in a two-dimensional domain $\Omega=(0,1)^2$. The parameters are taken as: $\gamma=1$, $\chi=1$, $\theta=1$, $\mu=1$, and $\alpha=1$. The homogeneous boundary conditions are imposed for both $\rho$ and $\phi$. 
The initial data are prescribed as follows:
\[
\left\{
\begin{aligned}
\rho^0 (x,y) &=1000{\rm e}^{-100\left[(x-\frac{1}{2})^2+(y-\frac{1}{2})^2\right]},\\
\phi^0 (x,y) &={\rm e}^{-100\left[(x-\frac{1}{2})^2+(y-\frac{1}{2})^2\right]},\\
\end{aligned}
\right.
\]
which mimics concentrated living organisms and chemical signals of high peak values initially distribute at the center of the domain.  
\begin{figure}[h]
\centering
\includegraphics[scale=.5]{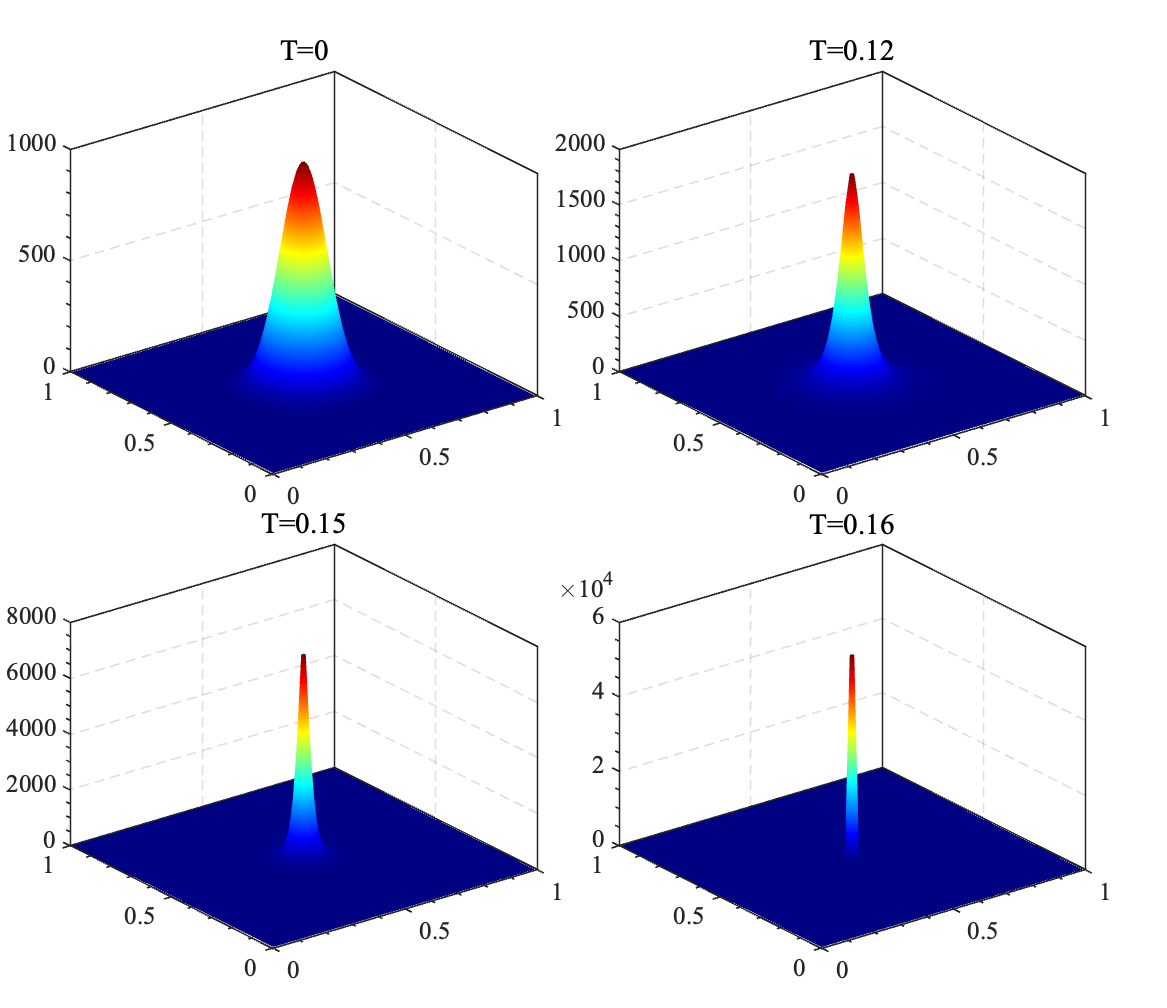}
\caption{Evolution of density $\rho$ at time instants: $T=0$, $0.12$, $0.15$, and $0.16$, with $\Delta t=10^{-5}$ and $h=10^{-2}$.}
\label{f:bu1}
\end{figure}
According to the mathematical analysis in~\cite{HerreroVelazquez_ASNS1997}, the analytic solution to the classical PKS system is expected to develop a blowup at the center of the domain in finite time, with initial mass $\| \rho^0\|_1\approx 31.4159>8\pi$. Indeed, Figure~\ref{f:bu1} displays numerical results on singularity formation at a sequence of time instants: $T=0$, $0.12$, $0.15$, and $0.16$, with a time step size $\Delta t=10^{-5}$. It is observed that the numerical solution of the living organism density evolves into blowup at the center of the domain,  leading to a dwindling support for the organism density.

\begin{figure}[h]
\centering
\includegraphics[scale=.5]{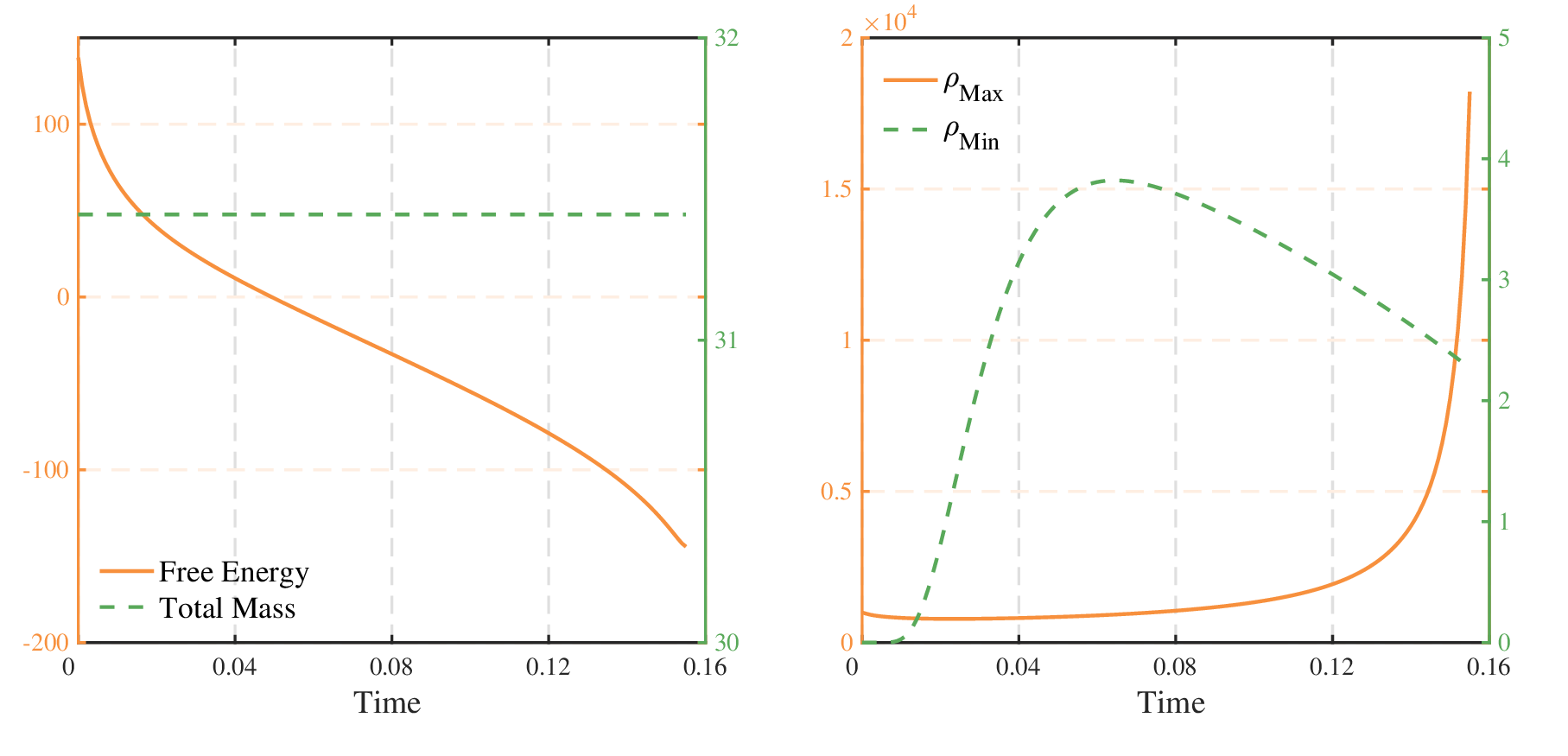}
\caption{Time evolution of the discrete energy $F_h$, mass of $\rho$, and the maximum and minimum values of $\rho$ over the computational mesh, with $\Delta t=10^{-5}$.}
\label{f:MassEnergyKS1}
\end{figure}
We next demonstrate the robustness of the proposed numerical scheme in preserving the desired properties in such blowup evolution, which poses a challenging task on numerical stability. With homogeneous Neumann boundary conditions, the physical system possesses mass conservation and free-energy dissipation. From the left panel of Figure~\ref{f:MassEnergyKS1},  one can see that the free energy~\reff{EnergyKS1} monotonically decreases and the total mass of $\rho$ remains constant robustly as time evolves. The right panel of Figure~\ref{f:MassEnergyKS1} displays the time evolution of $\rho_{\rm Min}:={\rm Min}_{i,j}~\rho_{i,j}$ and $\rho_{\rm Max}:={\rm Max}_{i,j}~\rho_{i,j}$, the minimum and maximum values of $\rho$ over the computational mesh, respectively. It is observed that the proposed numerical scheme is positivity-preserving and the maximum value of $\rho$ grows exponentially as time evolves in the singularity formation.

\subsubsection{Asymmetric initial data on a square}
The proposed second-order accurate scheme (\ref{PKS2nd_1}-\ref{PKS2nd_2}) is applied to probe blowup formation away from the peak position of initial densities. We consider the same IBVPs as in Section~\ref{ss:sym}, while the center of initial data is shifted to $(0.75,0.75)$:
\[
\left\{
\begin{aligned}
\rho^0 (x,y) &=1000 {\rm e}^{-100\left((x-0.75)^2+(y-0.75)^2\right)},\\
\phi^0 (x,y) &={\rm e}^{-100\left((x-0.75)^2+(y-0.75)^2\right)}.\\
\end{aligned}
\right.
\]
Again, the solution of $\rho$ may blow up in a finite time due to the fact that $\| \rho^0\|_1\approx 31.4033>8\pi$~\cite{HorstmannWang_2001EJAM}.

\begin{figure}[h]
\centering
\includegraphics[scale=.5]{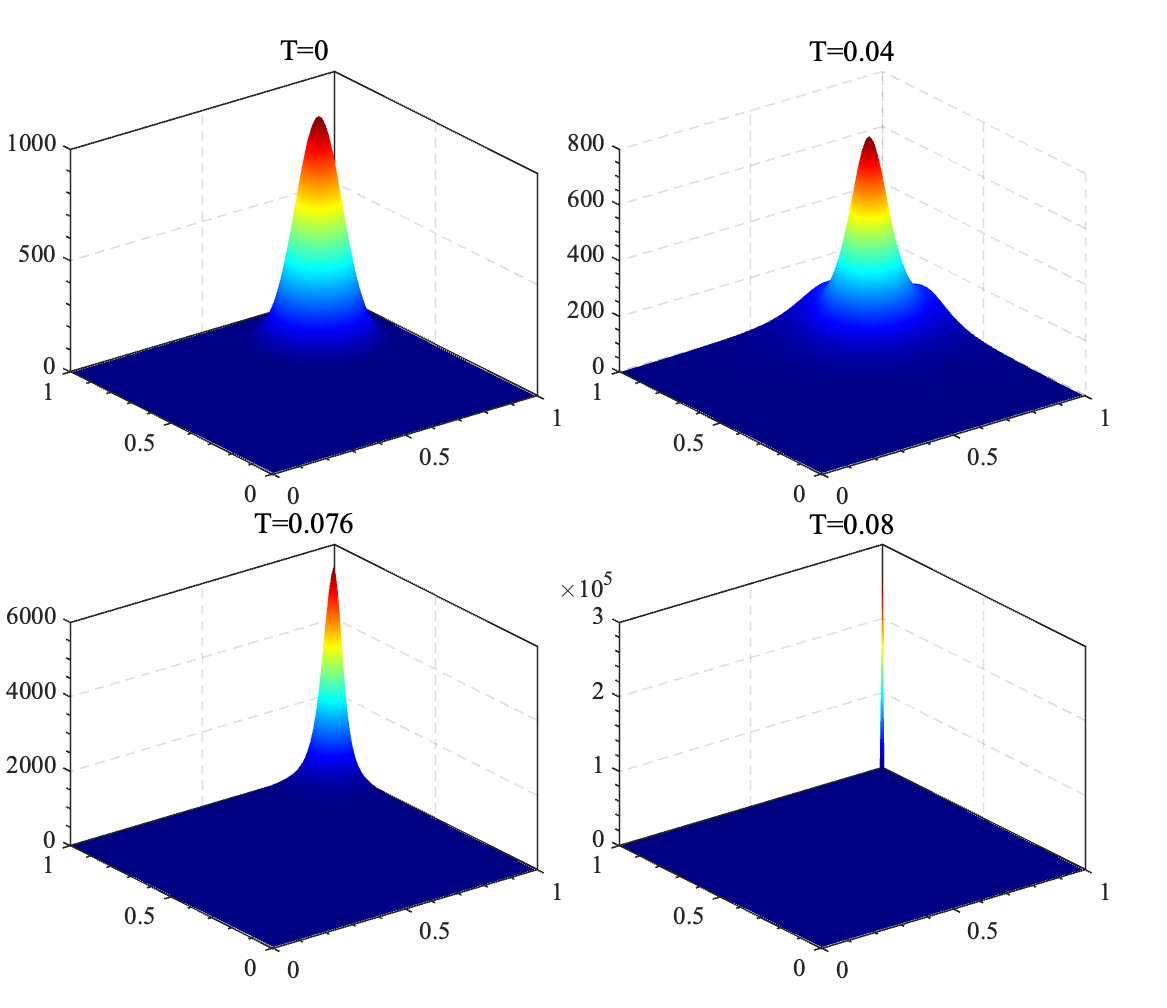}
\caption{Time evolution of $\rho$ at a sequence of time instants: $T=0$, $0.04$, $0.076$, and $0.08$, with $\Delta t=10^{-5}$ and $h=1/100$.}
\label{f:bu2}
\end{figure}
Figure~\ref{f:bu2} displays spatial density profiles of living organisms at four different time instants. It has been proved in \cite{HerreroVelazquez_ASNS1997} that the solution is expected to blow up at the boundary of the domain in this case. It is clearly observed that the behavior of the computed solution matches our expectation: the living organisms first move towards the boundary and then concentrate due to zero-flux boundary conditions, eventually forming solution blow up at the corner.
\begin{figure}[h]
\centering
\includegraphics[scale=.5]{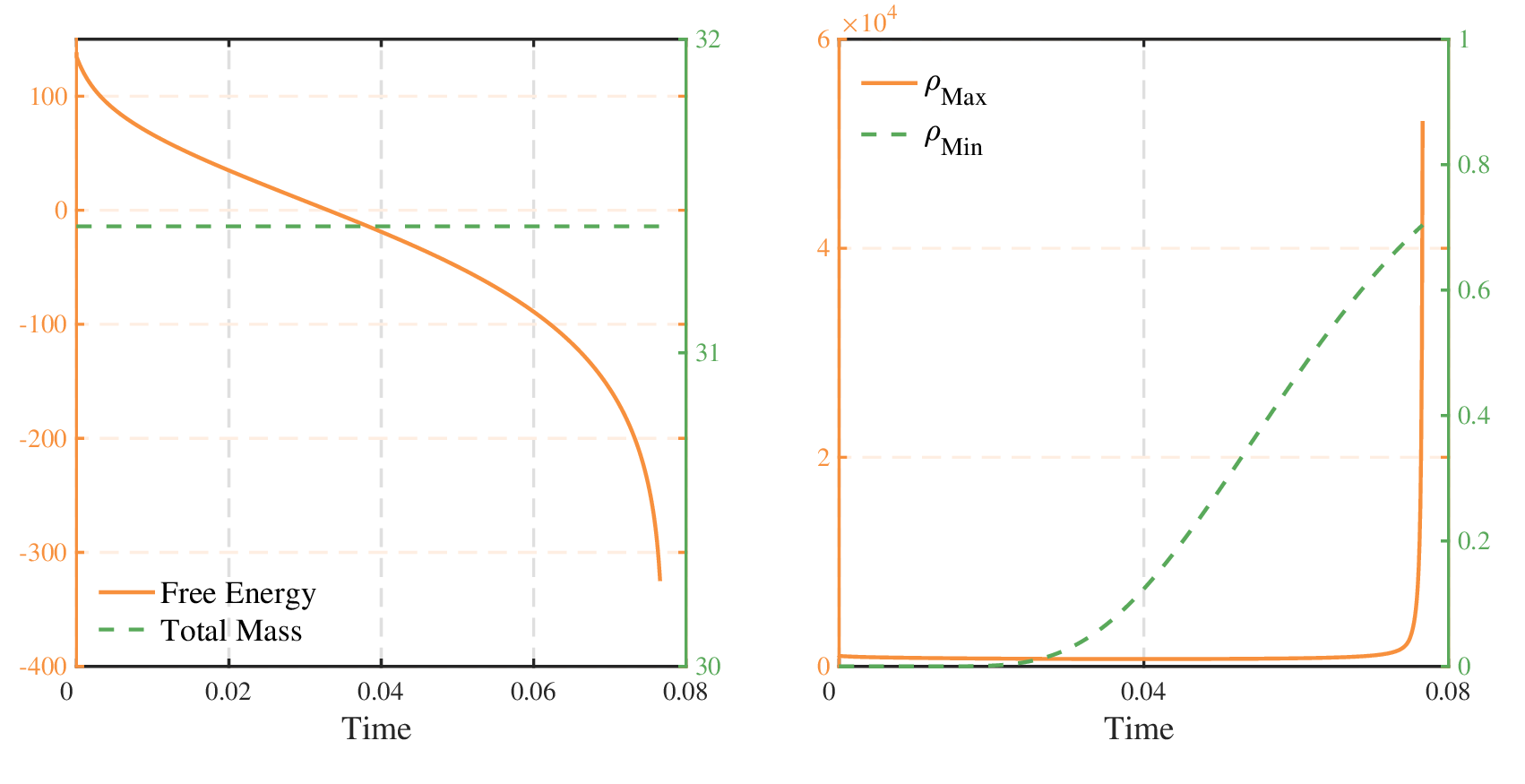}
\caption{Time evolution of the discrete energy $F_h$, mass of $\rho$, and the minimum and maximum values of $\rho$ over the computational mesh, with $\Delta t=10^{-5}$.
}
\label{f:MassEnergyKS2}
\end{figure}
Similar to the previous example, the free energy~\reff{EnergyKS1} depicted in Figure~\ref{f:MassEnergyKS2} gradually decreases over time and the total mass of $\rho$ remains at a constant value perfectly. Regarding the energy functional $F_h$, a significant decline in magnitude occurs before blowup and continuing decrease can be seen during the singularity formation.  Additionally, the density of living organism always remains positive, and the maximum value of density increases exponentially over time, indicating a solution blowup.

\section{Conclusions}
This work has proposed a novel second-order accurate numerical scheme for the PKS system with various mobilities for the description of chemotaxis. The variational structure of the PDE system has been used to facilitate the numerical design. The singular part in the chemical potential is discretized by a modified Crank-Nicolson approach, which leads to a nonlinear and singular numerical system. 
The unique solvability and positivity-preserving property have been theoretically established, in which the convexity of the nonlinear and singular term plays an important role. Moreover, a careful nonlinear analysis has proved a dissipation property of the original free energy functional, instead of a modified energy energy reported in many existing works for a multi-step numerical scheme. This makes the original energy stability analysis remarkable for the second-order discretization. In addition, this study has provided an optimal rate convergence analysis and error estimate for the proposed second order scheme, in which several highly non-standard techniques have been included. With a careful linearization expansion, the higher-order asymptotic expansion (up to fourth order accuracy in both time and space) has been performed. In turn, we are able to derive a rough error estimate, so that the $\ell^\infty$ bound for the density variable, as well as its temporal derivative, becomes available. Subsequently, a refined error estimate has been performed and the desired convergence estimate for $\rho$ is accomplished, in the $\ell^\infty(0,T;l^2)\cap \ell^2(0,T;H^1_h)$ norm. A few numerical results have confirmed the accuracy and robustness of the numerical scheme in preserving desired properties in simulations of chemotaxis.

\appendix
\numberwithin{equation}{section}
\makeatletter
\newcommand{\section@cntformat}{Appendix\thesection:\ }
\makeatother
\section[Appendix]{ Proof of Lemma \ref{Lem:CA1}}\label{Ap:A}
By the fact that $S^{n+\frac12} = \ln \rho^{n+1} + \frac{5 \rho^n}{6 \rho^{n+1}} - \frac{\rho^n}{6 ( \rho^{n+1} )^2 }- \frac23$, we see that the term $\td{S}^{n+\frac12}$ could be decomposed into two parts: $\td{S}^{n+\frac12} = \td{S}^{n+\frac12}_1+\td{S}^{n+\frac12}_2$, where
\begin{equation}\label{ApA:eq1}
\begin{aligned}
\td{S}^{n+\frac12}_1&=\check{S}^{n+\frac12}- \Big(\ln \crho^{n+1} + \frac{5\rho^n}{6\crho^{n+1}}-\frac{(\rho^n)^2}{6(\crho^{n+1})^2} - \frac23 \Big) 
 =\Big( \frac{5}{6\crho^{n+1}}-\frac{\crho^{n}+\rho^n}{6(\crho^{n+1})^2} \Big)\td \rho^n,\\
\td{S}^{n+\frac12}_2&=\ln \crho^{n+1}  + \frac{5\rho^n}{6\crho^{n+1}}-\frac{(\rho^n)^2}{6(\crho^{n+1})^2}- \frac23- S^{n+\frac12} 
= \left[ \frac{1}{\xi^{n+1}}-\frac{5\rho^n}{6(\xi^{n+1})^2}+\frac{(\rho^n)^2}{3(\xi^{n+1})^3}\right]\td \rho^{n+1},\\
\end{aligned}
\end{equation}
and $\xi^{n+1}$ is between $\crho^{n+1}$ and $\rho^{n+1}$, with an application of intermediate value theorem.

Recalling the phase separation property \reff{ConstSols:eq1}, regularity requirement \reff{ConstSols:eq2}, and a-priori $\|\cdot\|_\infty$ estimate  \reff{presolns:eq2} at the previous time step, we have
\[
\bigg|\frac{5}{6\crho^{n+1} }-\frac{\crho^{n} +\rho^n}{6(\crho^{n+1})^2}\bigg|\leq \max\left(\frac{5}{6\crho^{n+1}},~~ \frac{\crho^{n}+\rho^n}{6(\crho^{n+1})^2}\right)\leq \max\left(\frac{5}{6\ve^*},~~\frac{C^*+\ckC_0}{6(\ve^*)^2} \right):=\hat{C}_2 .
\]
Therefore, the following inequality is available:
\begin{equation}\label{ApA:eq2}
\begin{aligned}
\langle \td{S}^{n+\frac12}_1 , \td \rho^{n+1} \rangle =&h^3\sum_{i,j,k=1}^N
\Big( \frac{5}{6\crho^{n+1}_{i,j,k} }-\frac{\crho^{n}_{i,j,k} +\rho^n_{i,j,k} }{6(\crho^{n+1}_{i,j,k})^2} \Big) \td \rho^n_{i,j,k} \cdot \td \rho^{n+1}_{i,j,k}  
\ge 
-\hat{C}_2 h^3\sum_{i,j,k=1}^N|\td \rho^n_{i,j,k}  |\cdot|\td \rho^{n+1}_{i,j,k}  |.
\end{aligned}
\end{equation}

In terms of $\td{S}^{n+\frac12}_2$, we begin with the following observation:
\begin{equation}
\begin{aligned}
\frac{1}{\xi^{n+1} }-\frac{5\rho^n}{6(\xi^{n+1} )^2}+\frac{(\rho^n)^2}{3(\xi^{n+1} )^3}
=&\frac{2(\rho^n -\frac{5}{4}\xi^{n+1}  )^2+\frac{23}{8}(\xi^{n+1} )^2}{6(\xi^{n+1} )^3}
\ge  \frac{23}{48\xi^{n+1} }.
\end{aligned}
\end{equation}
This in turn implies that
\begin{equation}\label{ApA:eq3}
\langle \td{S}^{n+\frac12}_2 , \td \rho^{n+1} \rangle \geq h^3\sum_{i,j,k=1}^N \frac{23}{48\xi^{n+1}_{i,j,k} } |\td \rho^{n+1}_{i,j,k}|^2.
\end{equation}
Its combination with \reff{ApA:eq2} yields
\begin{equation}\label{ApA:eq4}
\begin{aligned}
  &
\langle \td{S}^{n+\frac12} , \td \rho^{n+1} \rangle \geq h^3\sum_{i,j,k=1}^N \Big(\frac{23}{48\xi^{n+1}_{i,j,k} } |\td \rho^{n+1}_{i,j,k}|^2- \hat{C}_2 |\td \rho^n_{i,j,k}  |\cdot|\td \rho^{n+1}_{i,j,k}  | \Big) ,  \quad
\mbox{and}
\\
  &
\langle \gamma \td{S}^{n+\frac12} + \td \psi^n , \td \rho^{n+1} \rangle \geq h^3\sum_{i,j,k=1}^N \Big(\frac{23\gamma }{48\xi^{n+1}_{i,j,k} } |\td \rho^{n+1}_{i,j,k}|^2
- ( \hat{C}_2\gamma  |\td \rho^n_{i,j,k}  | + | \td \psi^n_{i,j,k} | ) |\td \rho^{n+1}_{i,j,k}  | \Big) .
\end{aligned}
\end{equation}

At a fixed grid point $(i,j,k)$ that is not in $\K$, i.e., $0<\rho^{n+1}_{i,j,k}<2C^*+1$, the following estimate is available:
\[
\frac{1}{\xi^{n+1}_{i,j,k} }\geq \min \Big(\frac{1}{\crho^{n+1}_{i,j,k} }, ~\frac{1}{\rho^{n+1}_{i,j,k} }\Big) \geq \frac{1}{2C^*+1}.
\]
Subsequently, the following inequality is valid for $0<\rho^{n+1}_{i,j,k}<2C^*+1$:
\begin{equation}\label{ApA:eq5}
\begin{aligned}
&\frac{23\gamma }{48\xi^{n+1}_{i,j,k} }|\td \rho^{n+1}_{i,j,k} |^2 - (\hat{C}_2\gamma  |\td \rho^n_{i,j,k}| + | \td \psi^n_{i,j,k} | ) |\td \rho^{n+1}_{i,j,k}|\\
&\ge  \frac{23\gamma }{48(2C^*+1)}|\td \rho^{n+1}_{i,j,k}|^2 - \Big( \frac{23\gamma }{96(2C^*+1)}|\td \rho^{n+1}_{i,j,k}|^2+\frac{48 \gamma \hat{C}_2^2(2C^*+1)}{23}|\td \rho^n_{i,j,k}|^2
\\
  &  \qquad \qquad \qquad \qquad
+ \frac{48 (2C^*+1)}{23\gamma }|\td \psi^n_{i,j,k}|^2  \Big) \\
&= \frac{23\gamma }{96(2C^*+1)}|\td \rho^{n+1}_{i,j,k}|^2
- \frac{48\gamma \hat{C}_2^2(2C^*+1)}{23}|\td \rho^n_{i,j,k}|^2 - \frac{48 (2C^*+1)}{23\gamma }|\td \psi^n_{i,j,k}|^2 .
\end{aligned}
\end{equation}

On the other hand, if $(i,j,k)$ is in $\K$, i.e., $\rho^{n+1}_{i,j,k}>2C^*+1$, we have $\td \rho^{n+1}_{i,j,k}=\crho^{n+1}_{i,j,k}-\rho^{n+1}_{i,j,k}<0$, so that the following inequalities are valid:
\[
\begin{aligned}
&\crho^{n+1}_{i,j,k}\leq C^*\leq \frac{C^*}{2C^*+1}(2C^*+1)\leq \frac{C^* \rho^{n+1}_{i,j,k}}{2C^*+1},\\
&|\td \rho^{n+1}_{i,j,k} |=|\crho^{n+1}_{i,j,k}-\rho^{n+1}_{i,j,k}|\geq|\rho^{n+1}_{i,j,k}|- \frac{C^* \rho^{n+1}_{i,j,k}}{2C^*+1}\geq \frac{(C^* +1)}{2C^*+1}\rho^{n+1}_{i,j,k},\\
&\frac{|\td \rho^{n+1}_{i,j,k} |}{\xi^{n+1}_{i,j,k}}\geq\frac{|\td \rho^{n+1}_{i,j,k} |}{\rho^{n+1}_{i,j,k}}\geq \frac{(C^* +1)}{2C^*+1}.
\end{aligned}
\]
Subsequently, the following inequality could be derived:
\begin{equation}\label{ApA:eq6}
\begin{aligned}
 &\frac{23\gamma }{48\xi^{n+1}_{i,j,k} } |\td \rho^{n+1}_{i,j,k} |^2 - ( \hat{C}_2\gamma  |\td \rho^n_{i,j,k}| + | \td \psi^n_{i,j,k} | )  |\td \rho^{n+1}_{i,j,k}| \\
&\geq \frac{23(C^* +1)\gamma }{48(2C^*+1)}|\td \rho^{n+1}_{i,j,k} |
 - ( \gamma \hat{C}_2 \cdot C(\dt^{\frac{9}{4}}+h^{\frac{9}{4}}) + h) |\td \rho^{n+1}_{i,j,k} | \\
&\geq \frac{\gamma }{6} |\td \rho^{n+1}_{i,j,k} |\geq \frac{\gamma }{6}(C^*+1)\geq \frac{C^*\gamma }{6} ,
\end{aligned}
\end{equation}
where the $\|\cdot\|_{\infty}$ estimate \reff{presolns:eq1} and the assumption that $\| \td \psi^n \|_\infty \le h$ have been recalled in the second step, and the fact that $\frac{23(C^* +1)\gamma }{48(2C^*+1)}-( \gamma \hat{C}_2 \cdot C(\dt^{\frac{9}{4}}+h^{\frac{9}{4}}) + h) \geq \frac{\gamma }{6}$ has been used in the last step.

As a result, a substitution of the point-wise inequalities \reff{ApA:eq5} and \reff{ApA:eq6} into \reff{ApA:eq4} results in the desired estimate \reff{Lemr:eq1}, by taking $\ckC_2=\frac{48 \hat{C}_2^2(2C^*+1)}{23\gamma }$. Moreover, if $| \K |=0$, the improved nonlinear estimate in \reff{Lemr:eq2} could be derived, based on \eqref{ApA:eq6}, by taking $\ckC_3 = \frac{23\gamma }{96(2C^*+1)}$. Therefore, the proof of Lemma \ref{Lem:CA1} is completed.

\section{Proof of Proposition~\ref{prop:CA2}}\label{Ap:B} 
By the numerical error expansion formula~\eqref{ErrFun:eq1}, the following decomposition is available for $\td{S}^{n+\frac12}$, at a point-wise level:
\begin{equation}
\begin{aligned}
  &
 \td{S}^{n+\frac12} =  J_1 + J_2 + J_3 + J_4 + J_5 ,  \quad  \mbox{with}
\\
  &
  J_1 = \ln (\crho^{n+1}) - \ln (\rho^{n+1}) , \, \, \,
  J_2 = - \frac{1}{2 \rho^{n+1}} ( \tilde{\rho}^{n+1}- \tilde{\rho}^n) , \, \, \,
  J_3 =  \frac{\tilde{\rho}^{n+1}}{2 \crho^{n+1} \rho^{n+1}} (\crho^{n+1} - \crho^n) ,
\\
   &
   J_4 = - \frac{\crho^{n+1} - \crho^n + \rho^{n+1}-\rho^n}{6 ( \rho^{n+1})^2 }
   (\tilde{\rho}^{n+1} - \tilde{\rho}^n)  , \, \, \,
   J_5 =  \frac{( \crho^{n+1} + \rho^{n+1} ) \tilde{\rho}^{n+1} }{6 ( \crho^{n+1})^2 ( \rho^{n+1} )^2 }
      (\crho^{n+1}-\crho^n)^2 .
\end{aligned}
  \label{prop CA2-1}
\end{equation}
Meanwhile, at each cell, from $(i,j,k) \to (i+1,j,k)$, the following expansion identity is always valid:
\begin{equation}
  D_x (fg )_{i+\hf, j,k} = ( A_x f )_{i+\hf,j,k} \cdot ( D_x g )_{i+\hf,j,k}
  + ( A_x g )_{i+\hf,j,k} \cdot ( D_x f )_{i+\hf,j,k} .
  \label{prop CA2-2}
\end{equation}

We first look at the $D_x J_3$ term. Because of the point-wise bounds, $\ve^* \le \crho^{n+1} \le  C^*$, $\frac{\ve^*}{2} \le \rho^{n+1} \le \ckC_0$, and $\|\nabla_h\rho^{n+1} \|_\infty \le \tdC_0$ as given by~\eqref{ConstSols:eq1}, \eqref{propr:eq14}, and \reff{propr:eq15}, respectively,  the following estimates could be derived:
\begin{align}
  &
  0 < A_x ( \frac{1}{\crho^{n+1}} ) \le  ( \ve^* )^{-1} , \, \, \,
  0 < A_x ( \frac{1}{\rho^{n+1}} ) \le  2 ( \ve^* )^{-1} , \, \, \,
  0 < A_x ( \frac{1}{\crho^{n+1} \rho^{n+1}} ) \le  2 ( \ve^* )^{-2} ,   \label{prop CA2-3-1}
\\
  &
  | D_x ( \frac{1}{\crho^{n+1}} ) | \le ( \ve^* )^{-2} | D_x \crho^{n+1} |  \le C^* ( \ve^* )^{-2} , \, \,
  | D_x ( \frac{1}{\rho^{n+1}} ) | \le 4 ( \ve^* )^{-2} | D_x \rho^{n+1} |  \le 4 \td C_0 ( \ve^* )^{-2} ,
   \label{prop CA2-3-2}
\\
  &
   | D_x ( \frac{1}{\crho^{n+1} \rho^{n+1}} ) |
   \le A_x ( \frac{1}{\crho^{n+1}} )  \cdot   | D_x ( \frac{1}{\rho^{n+1}} ) |
   +  A_x ( \frac{1}{\rho^{n+1}} )  \cdot   | D_x ( \frac{1}{\crho^{n+1}} ) |   \nonumber
\\
  &  \qquad \qquad \qquad \quad
  \le ( \ve^* )^{-3} ( 2 C^* + 4 \td C_0 ) ,   \label{prop CA2-3-3}
\end{align}
in which inequalities \eqref{ConstSols:eq2} and \eqref{propr:eq15} have also been applied. Subsequently, a further application of~\eqref{prop CA2-2} leads to
\begin{equation}
\begin{aligned}
 \Big| A_x  \Big( \frac{\tilde{\rho}^{n+1}}{2 \crho^{n+1} \rho^{n+1}} \Big) \Big|
  & \le \frac12 \cdot 2 ( \ve^* )^{-2} A_x | \tilde{\rho}^{n+1} |
 = ( \ve^* )^{-2} A_x | \tilde{\rho}^{n+1} |  ,
\\
   \Big| D_x  \Big( \frac{\tilde{\rho}^{n+1}}{2 \crho^{n+1} \rho^{n+1}} \Big) \Big|
   &\le \frac12 A_x ( \frac{1}{\crho^{n+1} \rho^{n+1}} ) \cdot | D_x \tilde{\rho}^{n+1} |
   + \frac12 | D_x ( \frac{1}{\crho^{n+1} \rho^{n+1}}) | \cdot | A_x \tilde{\rho}^{n+1} |
\\
  & \le 
    ( \ve^* )^{-2} | D_x  \tilde{\rho}^{n+1} |
    +  ( \ve^* )^{-3} ( C^* + 2 \td C_0 )  A_x | \tilde{\rho}^{n+1} | .
\end{aligned}
   \label{prop CA2-4}
\end{equation}
Meanwhile, the regularity assumption~\eqref{ConstSols:eq2} (for the constructed profile $
\crho$) implies that
\begin{equation}
  | \crho^{n+1} - \crho^n | \le C^* \dt ,  \quad
  | D_x ( \crho^{n+1} - \crho^n ) | \le C^* \dt , \quad
  \mbox{at a point-wise level} .
  \label{prop CA2-5}
\end{equation}
Consequently, a combination of~\eqref{prop CA2-4} and \eqref{prop CA2-5} reveals that
\begin{equation}
\begin{aligned}
  | D_x J_3 |
  \le & \Big| A_x  \Big( \frac{\tilde{\rho}^{n+1}}{2 \crho^{n+1} \rho^{n+1}} \Big) \Big|
   \cdot | D_x ( \crho^{n+1} - \crho^n ) |
   + \Big| D_x  \Big( \frac{\tilde{\rho}^{n+1}}{2 \crho^{n+1} \rho^{n+1}} \Big) \Big|
   \cdot   A_x | \crho^{n+1} - \crho^n |
\\
  \le &
   ( \ve^* )^{-2} C^* \dt \Big( | D_x  \tilde{\rho}^{n+1} |
    +  ( ( \ve^* )^{-1} ( C^* + 2 \td C_0 )  + 1 ) A_x | \tilde{\rho}^{n+1} | \Big) .
\end{aligned}
  \label{prop CA2-6}
\end{equation}
Notice that we have dropped the $| \cdot |_{i+\hf, j,k}$ subscript notation on the right hand side terms, for simplicity of presentation, since all these inequalities are derived at a point-wise level.

Similar bounds could be derived for $| D_x J_4|$ and $|D_x J_5|$ terms: 
\begin{equation}
\begin{aligned}
  & 0 < A_x ( \frac{1}{(\rho^{n+1})^2} ) \le  4 ( \ve^* )^{-2} ,   \quad
  | D_x ( \frac{1}{(\rho^{n+1})^2} ) | \le 16 ( \ve^* )^{-3} | D_x \rho^{n+1} |  \le 16 \td C_0 ( \ve^* )^{-3} ,
\\
  &
  | \crho^{n+1} - \crho^n | ,  | D_x ( \crho^{n+1} - \crho^n ) | \le C^* \dt,  \quad
  | \rho^{n+1} - \rho^n | \le \td C_0 \dt,
\\
  &
   | A_x (\crho^{n+1} - \crho^n + \rho^{n+1} - \rho^n ) | \le ( C^* + \td C_0 ) \dt,  \quad
 |  D_x( \crho^{n+1} - \crho^n + \rho^{n+1} - \rho^n ) | \le C^* \dt + 2 \td C_0 ,
\\
  &
   | A_x ( \frac{\crho^{n+1} - \crho^n + \rho^{n+1} - \rho^n}{( \rho^{n+1} )^2 } ) |
   \le ( C^* + \td C_0 ) \dt \cdot 4 ( \ve^* )^{-2}  = 4 ( \ve^* )^{-2} ( C^* + \td C_0 ) \dt ,
\\
  &
 | D_x ( \frac{\crho^{n+1} - \crho^n + \rho^{n+1} - \rho^n}{( \rho^{n+1} )^2 } ) |
   \le ( C^* + \td C_0 ) \dt \cdot 16 \td C_0 ( \ve^* )^{-3}
   + ( C^* \dt + 2 \td C_0 ) \cdot 4 ( \ve^* )^{-2}   
\\
  &  \qquad \qquad \qquad \qquad \qquad
  \le  4( \ve^* )^{-2}\dt  \Big( C^* + 4\td C_0 ( \ve^* )^{-1}( \td C_0 + C^*)\Big)
   + 8 ( \ve^* )^{-2}  \td C_0  
  \le  8 ( \ve^* )^{-2}  \td C_0  + 1 ,
\\
  &
  | D_x J_4 |
  \le  \frac16 | A_x ( \frac{\crho^{n+1} - \crho^n + \rho^{n+1} - \rho^n}{( \rho^{n+1} )^2 } ) |
   \cdot | D_x ( \td \rho^{n+1} - \td \rho^n ) |
\\
  &  \qquad \qquad \qquad  \quad
   + \frac16 | D_x ( \frac{\crho^{n+1} - \crho^n + \rho^{n+1} - \rho^n}{( \rho^{n+1} )^2 } ) |
   \cdot   A_x | \td \rho^{n+1} - \td \rho^n |
\\
  & \qquad \qquad
   \le \frac23 ( \ve^* )^{-2} ( C^* + \td C_0 ) \dt ( | D_x  \tilde{\rho}^{n+1} | +  | D_x  \tilde{\rho}^n | )
   + \frac16 ( 8 ( \ve^* )^{-2}  \td C_0  + 1 ) ( A_x | \tilde{\rho}^{n+1} | + A_x | \tilde{\rho}^n | ) 
   &
\\
   & 0 \le ( \crho^{n+1} - \crho^n )^2 \le ( C^* )^2 \dt^2 ,
\\
  &
   | D_x ( ( \crho^{n+1} - \crho^n )^2 )  | \le 2 \| D_x ( \crho^{n+1} - \crho^n ) \|_\infty
   \cdot \| \crho^{n+1} - \crho^n  \|_\infty \le 2 ( C^* )^2 \dt^2 ,
\\
  &
  0 < A_x ( \frac{1}{(\crho^{n+1})^2 (\rho^{n+1})^2 } ) \le  ( \ve^* )^{-2} \cdot 4 ( \ve^* )^{-2}
  = 4 ( \ve^* )^{-4} ,
\\
  &
  | D_x ( \frac{1}{(\crho^{n+1})^2 (\rho^{n+1})^2} ) | \le 8 (  \ve^* )^{-5} (  | D_x \crho^{n+1} |
  +2 | D_x \rho^{n+1} | ) \le 8 (C^* +2\td C_0) ( \ve^* )^{-5} ,
\\
  &
 | A_x ( \crho^{n+1} + \rho^{n+1} )  | \le C^* + \ckC_0 ,  \quad
  | D_x ( \crho^{n+1} + \rho^{n+1} )  | \le | D_x \crho^{n+1} | + | D_x\rho^{n+1}  |
  \le C^* + \td C_0 ,
\\
  &
 | A_x ( \frac{\crho^{n+1} + \rho^{n+1}}{( \crho^{n+1} )^2 (\rho^{n+1})^2 } ) |
   \le 4 ( \ve^* )^{-4} ( C^* + \ckC_0 )  ,
\\
  &
 | D_x ( \frac{\crho^{n+1} + \rho^{n+1} }{( \crho^{n+1} )^2 (\rho^{n+1})^2 } ) |
   \le ( C^* + \ckC_0 ) \cdot 8 (C^* +2\td C_0) ( \ve^* )^{-5}
   + ( C^* + \td C_0 ) \cdot 4 ( \ve^* )^{-4}
\\
  &  \qquad \qquad \qquad \qquad  \quad
 \le 8 ( \ve^* )^{-5}  ( C^* + \ckC_0 ) ( C^*+ 2\td C_0 + 1) ,
\\
  &
 | A_x ( \frac{( \crho^{n+1} + \rho^{n+1} ) ( \crho^{n+1} - \crho^n )^2}{( \crho^{n+1} )^2 (\rho^{n+1})^2 } ) |
   \le 4 ( \ve^* )^{-4} ( C^* + \ckC_0 )  (C^*)^2 \dt^2 ,
\\
&
 | D_x ( \frac{( \crho^{n+1} + \rho^{n+1} ) ( \crho^{n+1} - \crho^n )^2}{( \crho^{n+1} )^2 (\rho^{n+1})^2 } ) |
\\
  & \qquad
 \le  8 ( \ve^* )^{-5}  ( C^* + \ckC_0 ) ( C^*+ 2\td C_0 + 1)  \cdot  (C^*)^2 \dt^2
   +  4 ( \ve^* )^{-4} ( C^* + \ckC_0 ) \cdot 2  (C^*)^2 \dt^2
\\
  & \qquad
  \le  8 ( \ve^* )^{-5} (C^*)^2 \dt^2 ( C^* + \ckC_0 ) ( C^*+ 2\td C_0 + 2),
\end{aligned}
  \label{prop CA2-7}
\end{equation}
\begin{equation}
\begin{aligned}
  &
  | D_x J_5 |
  \le \frac16 | A_x ( \frac{( \crho^{n+1} + \rho^{n+1} ) ( \crho^{n+1} - \crho^n )^2}{( \crho^{n+1} )^2 (\rho^{n+1})^2 } ) |   \cdot | D_x \td \rho^{n+1}  |
\\
  &  \qquad \qquad \qquad  \quad
  +\frac16  | D_x ( \frac{( \crho^{n+1} + \rho^{n+1} ) ( \crho^{n+1} - \crho^n )^2}{( \rho^{n+1} )^2 (\rho^{n+1})^2 } ) |  \cdot  A_x | \td \rho^{n+1} |
\\
  & \qquad \quad
  \le \frac23  ( \ve^* )^{-4} (C^*)^2 \dt^2 ( C^* + \ckC_0 )\Big(2(\ve^*)^{-1} ( C^*+ 2\td C_0 + 2)\cdot A_x | \td \rho^{n+1} | +
     | D_x \td \rho^{n+1}  | \Big)
\\
  & \qquad \quad
 \le  \dt ( | D_x  \tilde{\rho}^{n+1} | +  A_x | \tilde{\rho}^{n+1} | ).
\end{aligned}
  \label{prop CA2-8}
\end{equation}
Notice that the phase separation property~\eqref{ConstSols:eq1}, regularity assumption~\eqref{ConstSols:eq2} for the constructed profile, a-priori estimate \reff{presolns:eq2}-\reff{lemim}, and the rough $\| \cdot\|_\infty$ estimates \reff{propr:eq14}-\reff{propr:eq17-2} for the numerical solution, have been repeatedly applied in the above derivation.


In fact, the finite difference operations for the $J_3$, $J_4$ and $J_5$ terms could be viewed as higher-order perturbations in the nonlinear expansion of $D_x \td S^{n+\frac12}$, and the two leading terms, namely $D_x J_1$ and $D_x J_2$, turn out to play a dominant role in the nonlinear error estimate. Now we focus on the $D_x J_1$ term. Within a single mesh cell $(i,j,k) \to (i+1,j,k)$, the following expansion is straightforward, based on the mean value theorem:
\begin{align}
  D_x (  \ln \crho^{n+1} - \ln \rho^{n+1} )_{i+\hf,j,k}
  =& \frac{1}{h} ( \ln \crho_{i+1,j,k}^{n+1} - \ln \crho_{i,j,k}^{n+1} )
  -  \frac{1}{h} ( \ln \rho_{i+1,j,k}^{n+1} - \ln \rho_{i,j,k}^{n+1} )  \nonumber
\\
   =& \frac{1}{\xi_{\crho}} D_x \crho_{i+\hf, j,k}^{n+1}
   - \frac{1}{\xi_\rho} D_x \rho_{i+\hf, j,k}^{n+1}    \nonumber
\\
  =& \Big( \frac{1}{\xi_{\crho}} -  \frac{1}{\xi_\rho} \Big) D_x \crho_{i+\hf, j,k}^{n+1}
  + \frac{1}{\xi_\rho} D_x \tilde{\rho}_{i+\hf, j,k}^{n+1}  ,
  \label{nonlinear est-1-1}
  \end{align}
  where
\begin{align}
  \frac{1}{\xi_{\crho}}  = \frac{\ln \crho_{i+1,j,k}^{n+1} - \ln \crho_{i,j,k}^{n+1} }{ \crho_{i+1,j,k}^{n+1} -  \crho_{i,j,k}^{n+1} }, \quad &  \quad
  \frac{1}{\xi_\rho}  = \frac{\ln \rho_{i+1,j,k}^{n+1} - \ln \rho_{i,j,k}^{n+1} }{ \rho_{i+1,j,k}^{n+1} -  \rho_{i,j,k}^{n+1} } .
  \label{nonlinear est-1-2}
	\end{align}
Meanwhile, for any $a > 0$ and $b > 0$, a careful Taylor expansion of $\ln x$ around a middle point $x_0 = \frac{a+b}{2}$ reveals that
\begin{equation}
  \frac{\ln b - \ln a}{b-a} = \frac{1}{ x_0} + \frac{(b-a)^2}{12 x_0^3} + \frac{(b-a)^4}{160}
  \Big( \frac{1}{\eta_1^5} + \frac{1}{\eta_2^5} \Big) ,  \label{Taylor-1}
\end{equation}
in which $\eta_1$ is between $a$ and $x_0$, $\eta_2$ is between $x_0$ and $b$. It is observed that, only the even order terms appear in~\eqref{Taylor-1}, due to the symmetric expansion around $x_0 = \frac{a+b}{2}$. This fact will greatly simplify the nonlinear analysis. In turn, a more precise representation for $\frac{1}{\xi_{\crho}}$ and $\frac{1}{\xi_{\rho}}$, as given in~\eqref{nonlinear est-1-2}, becomes available:
\begin{equation}
\begin{aligned}
  &
  \frac{1}{\xi_{\crho}} = \frac{1}{A_x \crho_{i+\hf,j,k}^{n+1}}
  + \frac{h^2 ( D_x \crho_{i+\hf,j,k}^{n+1} )^2}{12 ( A_x \crho_{i+\hf,j,k}^{n+1} )^3}
   + \frac{h^4 ( D_x \crho_{i+\hf,j,k}^{n+1} )^4}{160}
  \Big( \frac{1}{\eta_{\crho,1}^5} + \frac{1}{\eta_{\crho, 2}^5} \Big) ,
\\
  &
  \frac{1}{\xi_{\rho}} = \frac{1}{A_x \rho_{i+\hf,j,k}^{n+1}}
  + \frac{h^2 ( D_x \rho_{i+\hf,j,k}^{n+1} )^2}{12 ( A_x \rho_{i+\hf,j,k}^{n+1} )^3}
   + \frac{h^4 ( D_x \rho_{i+\hf,j,k}^{n+1} )^4}{160}
  \Big( \frac{1}{\eta_{\rho,1}^5} + \frac{1}{\eta_{\rho, 2}^5} \Big) ,
\\
  &
  \mbox{$\eta_{\crho,1}$ is between $\crho_{i,j,k}^{n+1}$ and $A_x \crho_{i+\hf,j,k}^{n+1}$, $\eta_{\crho,2}$ is between $A_x \crho_{i+\hf,j,k}^{n+1}$ and $\crho_{i+1,j,k}^{n+1}$} ,
\\
  &
  \mbox{$\eta_{\rho,1}$ is between $\rho_{i,j,k}^{n+1}$ and $A_x \rho_{i+\hf,j,k}^{n+1}$, $\eta_{\rho,2}$ is between $A_x \rho_{i+\hf,j,k}^{n+1}$ and $\rho_{i+1,j,k}^{n+1}$} .
\end{aligned}
  \label{nonlinear est-1-3}
\end{equation}
Then we obtain a detailed expansion of $\frac{1}{\xi_{\crho}} - \frac{1}{\xi_{\rho}}$, which is needed in the analysis for~\eqref{nonlinear est-1-2}:
\begin{equation}
\begin{aligned}
  \frac{1}{\xi_{\crho}} - \frac{1}{\xi_{\rho}}  = &
  \frac{-( A_x \td \rho)_{i+\hf,j,k}^{n+1}}{A_x \crho_{i+\hf,j,k}^{n+1} \cdot A_x \rho_{i+\hf,j,k}^{n+1}}   + \frac{h^2 ( D_x \rho_{i+\hf,j,k}^{n+1} + D_x \crho_{i+\hf,j,k}^{n+1} ) D_x \td \rho_{i+\hf,j,k}^{n+1} }{12 ( A_x \rho_{i+\hf,j,k}^{n+1} )^3}
\\
  &
  - \frac{h^2 ( D_x \crho_{i+\hf,j,k}^{n+1} )^2 ( A_x \td \rho)_{i+\hf,j,k}^{n+1} }{12 ( A_x \crho_{i+\hf,j,k}^{n+1} )^3 ( A_x \rho_{i+\hf,j,k}^{n+1} )^3}
\\
  &  \qquad
  \cdot \Big(  ( A_x \crho_{i+\hf,j,k}^{n+1} )^2 + A_x \crho_{i+\hf,j,k}^{n+1}  \cdot A_x \rho_{i+\hf,j,k}^{n+1} + ( A_x \rho_{i+\hf,j,k}^{n+1} )^2 \Big)
\\
  &
  + \frac{h^4 ( D_x \crho_{i+\hf,j,k}^{n+1} )^4}{160}
  \Big( \frac{1}{\eta_{\crho,1}^5} + \frac{1}{\eta_{\crho, 2}^5} \Big)
  - \frac{h^4 ( D_x \rho_{i+\hf,j,k}^{n+1} )^4}{160}
  \Big( \frac{1}{\eta_{\rho,1}^5} + \frac{1}{\eta_{\rho, 2}^5} \Big) .
\end{aligned}
  \label{nonlinear est-1-4}
\end{equation}
On the other hand, because of the following bounds at $(i + \hf, j,k)$ and time step $t_{n+1}$:
\begin{equation}
\begin{aligned}
  &
   \ve^* \le A_x \crho \le C^* , \, \, \,  | D_x \crho | \le C^* , \quad
  \frac{\ve^*}{2} \le A_x \rho \le \ckC_0 , \,  \, \, | D_x \rho | \le \td C_0 ,
\\
  &
  \ve^* \le \eta_{\crho,1} , \, \eta_{\crho,2} \le C^* ,  \quad
  \frac{\ve^*}{2} \le \eta_{\rho,1} , \, \eta_{\rho,2} \le \ckC_0 ,
\end{aligned}
  \label{nonlinear est-1-5}
\end{equation}
the following estimates could be derived:
\begin{equation}
\begin{aligned}
  &
  \Big| \frac{ A_x \td \rho }{A_x \crho \cdot A_x \rho} \Big| \le  2 ( \ve^* )^{-2} | A_x \td \rho | ,
\\
  &
 \Big|  \frac{h^2 ( D_x \rho + D_x \crho ) D_x \td \rho }{12 ( A_x \rho )^3}  \Big|
  \le \frac{h^2}{12} \cdot 8 ( \ve^* )^{-3} \cdot ( C^* + \td C_0 ) | D_x \td \rho |
  = \frac{2 h^2}{3} ( \ve^* )^{-3} ( C^* + \td C_0 ) | D_x \td \rho |  ,
\\
  &
   \Big| \frac{h^2 ( D_x \crho )^2 ( A_x \td \rho) }{12 ( A_x \crho )^3 ( A_x \rho )^3}
  \cdot \Big(  ( A_x \crho )^2 + A_x \crho  \cdot A_x \rho + ( A_x \rho )^2 \Big) \Big|
\\
  &  \qquad
  \le \frac{h^2}{12} \cdot 8 ( \ve^* )^{-6} \cdot ( C^* +\ckC_0)^2 \cdot  (C^*)^2 | A_x \td \rho |
  = \frac23 ( \ve^* )^{-6} ( C^*+\ckC_0 )^2 (C^*)^2 h^2 | A_x \td \rho |  ,
\\
  &
   \Big| \frac{h^4 ( D_x \crho )^4}{160}
  \Big( \frac{1}{\eta_{\crho,1}^5} + \frac{1}{\eta_{\crho, 2}^5} \Big)  \Big|
  \le \frac{h^4}{160} \cdot ( C^*)^4 \cdot 2 (\ve^*)^{-5}
  = \frac{( C^*)^4 (\ve^*)^{-5} }{80} h^4 ,
\\
  &
  \Big| \frac{h^4 ( D_x \rho )^4}{160}
  \Big( \frac{1}{\eta_{\rho,1}^5} + \frac{1}{\eta_{\rho, 2}^5} \Big) \Big|
   \le \frac{h^4}{160} \cdot ( \td C_0 )^4 \cdot 64 (\ve^*)^{-5}
  = \frac{2 ( \td C_0 )^4 (\ve^*)^{-5} }{5} h^4 .
\end{aligned}
  \label{nonlinear est-1-6}
\end{equation}
Again, we have dropped the $| \cdot |_{i+\hf, j,k}^{n+1}$ script notation in the analysis, for simplicity of presentation, and all these inequalities are derived at a point-wise level. Subsequently, a substitution of \eqref{nonlinear est-1-6} into \eqref{nonlinear est-1-4} yields
\begin{equation}
  \Big| \frac{1}{\xi_{\crho}} - \frac{1}{\xi_{\rho}}  \Big|
  \le  ( 2 ( \ve^* )^{-2} + 1) |  A_x \td \rho_{i+\hf, j,k}^{n+1} |
   + h |  D_x \td \rho_{i+\hf, j,k}^{n+1} |  + \frac{2 ( \td C_0 +C^*)^4 (\ve^*)^{-5} }{5} h^4 ,
   \label{nonlinear est-1-7}
\end{equation}
provided that $h$ is sufficiently small.

In terms of the second expansion term on the right hand side of~\eqref{nonlinear est-1-1}, we observe an $O (h^2)$ Taylor expansion to obtain $\frac{1}{\xi_\rho}$, in a similar formula as in~\eqref{nonlinear est-1-3}:
 \begin{equation}
\begin{aligned}
  &
  \frac{1}{\xi_{\rho}} = \frac{1}{A_x \rho_{i+\hf,j,k}^{n+1}}
   + \frac{h^2 ( D_x \rho_{i+\hf,j,k}^{n+1} )^2}{24}
  \Big( \frac{1}{\eta_{\rho,3}^3} + \frac{1}{\eta_{\rho, 4}^3} \Big) ,
\\
  &
  \mbox{$\eta_{\rho,3}$ is between $\rho_{i,j,k}^{n+1}$ and $A_x \rho_{i+\hf,j,k}^{n+1}$, $\eta_{\rho,4}$ is between $A_x \rho_{i+\hf,j,k}^{n+1}$ and $\rho_{i+1,j,k}^{n+1}$} .
\end{aligned}
  \label{nonlinear est-2-1}
\end{equation}
Again, the remaining expansion terms could be bounded as follows
\begin{equation}
\begin{aligned}
  &
  \frac{\ve^*}{2} \le A_x \rho \le \ckC_0 , \,  \, \, | D_x \rho | \le \td C_0 ,  \quad
  \frac{\ve^*}{2} \le \eta_{\rho,3} , \, \eta_{\rho,4} \le \ckC_0 ,  \quad \mbox{so that}
\\
  &
  \Big| \frac{h^2 ( D_x \rho )^2}{24}
  \Big( \frac{1}{\eta_{\rho,3}^3} + \frac{1}{\eta_{\rho, 4}^3} \Big) \Big|
   \le \frac{h^2}{24} \cdot ( \td C_0 )^2 \cdot 16 (\ve^*)^{-3}
  = \frac{2 ( \td C_0 )^2 (\ve^*)^{-3} }{3} h^2 \le h ,
\end{aligned}
  \label{nonlinear est-2-2}
\end{equation}
provided that $h$ is sufficiently small.

As a result, a substitution of \eqref{nonlinear est-1-7}-\eqref{nonlinear est-2-2} into \eqref{nonlinear est-1-1} reveals the following fact, at a point-wise level:
\begin{equation}
\begin{aligned}
  &
  ( D_x J_1 )_{i+\hf, j,k} = \frac{D_x \td \rho_{i+\hf, j,k}^{n+1} }{A_x \rho_{i+\hf,j,k}^{n+1}}
  + \zeta_1^{n+\frac12} ,  \quad \mbox{with}
\\
  &
  | \zeta_1^{n+\frac12} | \le
  ( 2 ( \ve^* )^{-2} + 1)C^*  |  A_x \td \rho_{i+\hf, j,k}^{n+1} |
   + (C^* + 1) h |  D_x \td \rho_{i+\hf, j,k}^{n+1} |
   \\
   &\qquad\qquad+ \frac{2 C^*  ( \td C_0 +C^*)^4 (\ve^*)^{-5} }{5} h^4 .
\end{aligned}
  \label{nonlinear est-3}
\end{equation}
Notice that the $W_h^{1,\infty}$ bound for $\crho$, $\| D_x \crho^{n+1} \|_\infty \le C^*$, has been applied.

The analysis for the $D_x J_2$ part is more straightforward. An application of identity~\eqref{prop CA2-2} gives
\begin{equation}
  D_x J_2 = - \frac12 A_x ( \frac{1}{\rho^{n+1}} ) \cdot D_x ( \tilde{\rho}^{n+1}- \tilde{\rho}^n)
  - \frac12 D_x ( \frac{1}{\rho^{n+1}} ) \cdot A_x ( \tilde{\rho}^{n+1}- \tilde{\rho}^n) .
  \label{nonlinear est-4-1}
\end{equation}
The second term could be bounded as follows
\begin{equation}
\begin{aligned}
  &
  | D_x ( \frac{1}{\rho^{n+1}} ) | \le 4 \td C_0 ( \ve^* )^{-2} ,  \quad \mbox{so that}
\\
  &
  \Big| D_x ( \frac{1}{\rho^{n+1}} ) \cdot A_x ( \tilde{\rho}^{n+1}- \tilde{\rho}^n) \Big|
  \le 4 \td C_0 ( \ve^* )^{-2}  ( | A_x  \tilde{\rho}^{n+1} | + | A_x \tilde{\rho}^n | ) .
\end{aligned}
  \label{nonlinear est-4-2}
\end{equation}
In terms of the first part on the right hand side of~\eqref{nonlinear est-4-1}, we need to estimate the difference between $A_x ( \frac{1}{\rho^{n+1}} )$ and $\frac{1}{A_x \rho^{n+1}}$. Similarly, for any $a > 0$, $b > 0$, a careful Taylor expansion of $\frac{1}{x}$ around a middle point $x_0 = \frac{a+b}{2}$ reveals that
\begin{equation}
  \frac12 \Big( \frac{1}{a} + \frac{1}{b} \Big) = \frac{1}{x_0} + \frac{(b-a)^2}{8}
  \Big( \frac{1}{\eta_1^3} + \frac{1}{\eta_2^3} \Big) ,  \label{Taylor-2}
\end{equation}
in which $\eta_1$ is between $a$ and $x_0$, $\eta_2$ is between $x_0$ and $b$. In turn, by setting $a=\rho_{i,j,k}^{n+1}$, $b=\rho_{i+1,j,k}^{n+1}$, we see that
\begin{equation}
\begin{aligned}
  &
  A_x ( \frac{1}{\rho^{n+1}} )_{i+\hf, j,k} - \frac{1}{A_x \rho^{n+1}_{i+\hf, j,k} }
  = \frac{h^2 (D_x \rho_{i+\hf, j,k}^{n+1} )^2 }{8}
  \Big( \frac{1}{\eta_{\rho, 5}^3} + \frac{1}{\eta_{\rho, 6}^3} \Big) ,
\\
  &
  \mbox{$\eta_{\rho,5}$ is between $\rho_{i,j,k}^{n+1}$ and $A_x \rho_{i+\hf,j,k}^{n+1}$, $\eta_{\rho,6}$ is between $A_x \rho_{i+\hf,j,k}^{n+1}$ and $\rho_{i+1,j,k}^{n+1}$} .
\end{aligned}
  \label{nonlinear est-4-3}
\end{equation}
A bound for the remainder term is available:
\begin{equation}
\begin{aligned}
  &
  | D_x \rho^{n+1} | \le \td C_0 ,  \quad
  \frac{\ve^*}{2} \le \eta_{\rho, 5} , \, \eta_{\rho, 6} \le \ckC_ 0 ,  \quad \mbox{so that}
\\
  &
  \Big| \frac{h^2 (D_x \rho_{i+\hf, j,k}^{n+1} )^2 }{8}
  \Big( \frac{1}{\eta_{\rho, 5}^3} + \frac{1}{\eta_{\rho, 6}^3} \Big) \Big|
  \le \frac{h^2}{8} \cdot ( \td C_0 )^2 \cdot 16 ( \ve^* )^{-3} \le h ,
\end{aligned}
  \label{nonlinear est-4-4}
\end{equation}
provided that $h$ is sufficiently small. Therefore, a substitution of \eqref{nonlinear est-4-2}, \eqref{nonlinear est-4-3}, and \eqref{nonlinear est-4-4} into \eqref{nonlinear est-4-1} leads to
\begin{equation}
\begin{aligned}
   &
  ( D_x J_2 )_{i+\hf, j, k}
  = - \frac{D_x ( \tilde{\rho}^{n+1}- \tilde{\rho}^n)_{i+\hf, j,k} }{2 A_x \rho^{n+1}_{i+\hf, j,k} }
  + \zeta_2^{n+\frac12} ,  \quad \mbox{with}
\\
  &
  | \zeta_2^{n+\frac12} | \le 2 \td C_0 ( \ve^* )^{-2}  ( |A_x  \tilde{\rho}^{n+1}_{i+\hf, j,k} |
  + | A_x \tilde{\rho}^n_{i+\hf, j,k} | )
  + \frac{h}{2} ( | D_x  \tilde{\rho}^{n+1}_{i+\hf, j,k} |
  + | D_x \tilde{\rho}^n_{i+\hf, j,k} | )  .
\end{aligned}
  \label{nonlinear est-4-5}
\end{equation}

Finally, a combination of \eqref{prop CA2-6}, \eqref{prop CA2-7}, \eqref{prop CA2-8}, \eqref{nonlinear est-3}, and \eqref{nonlinear est-4-5} results in
\begin{equation}
\begin{aligned}
  &
  ( D_x \td S^{n+\frac12} )_{i+\hf, j,k} = \frac{D_x ( \td \rho^{n+1}  + \td \rho^n )_{i+\hf, j,k}  }{2 A_x \rho_{i+\hf,j,k}^{n+1}}
  + \zeta^{n+\frac12} ,
\\
  &
  | \zeta^{n+\frac12} | \le
  \breve{C}_1 (  A_x | \td \rho_{i+\hf, j,k}^{n+1} |  + A_x | \td \rho_{i+\hf, j,k}^n | )
   + ( \breve{C}_2 h  + \breve{C}_3 \dt ) ( |  D_x \td \rho_{i+\hf, j,k}^{n+1} |
   + |  D_x \td \rho_{i+\hf, j,k}^n | )
\\
  & \qquad \qquad
   + \frac{2 C^* ( \td C_0 +C^*)^4 (\ve^*)^{-5}}{5} h^4 ,
\end{aligned}
  \label{nonlinear est-5-1}
\end{equation}
provided that $\dt$ and $h$ are sufficiently small, with
\begin{equation}
\begin{aligned}
  &
  \breve{C}_1 = \frac16 ( 8 ( \ve^* )^{-2}  \td C_0  + 1 ) +  ( 2 ( \ve^* )^{-2} + 1)C^*
  + 2 \td C_0 ( \ve^* )^{-2} + 1 ,
\\
  &
  \breve{C}_2 = C^* + \frac32 , \quad
  \breve{C}_3 = \frac53 ( \ve^* )^{-2} ( C^* + \td C_0 ) +1 .
\end{aligned}
  \label{nonlinear est-5-2}
\end{equation}
In particular, it is observed that the following identity has played a crucial role in the combined form of~\eqref{nonlinear est-5-1}:
\begin{equation}
  D_x \td \rho^{n+1} - \frac12 D_x ( \td \rho^{n+1} - \td \rho^n )
  = \frac12 D_x ( \td \rho^{n+1} + \td \rho^n )  .
  \label{nonlinear est-5-3}
\end{equation}

  As a consequence, the point-wise estimate~\eqref{nonlinear est-5-1}  implies that
\begin{equation}
\begin{aligned}
\ciptwo{\hat{\rho}^{n+\frac{1}{2}} D_x \td S^{n+\frac12} }{ D_x ( \td\rho^{n+1} + \td \rho^n ) }
  =  & \Big\langle \frac{\hat{\rho}^{n+\frac12}}{2 \rho^{n+1}}  D_x ( \td \rho^{n+1} + \td \rho^n )  ,
   D_x ( \td\rho^{n+1} + \td \rho^n )  \Big\rangle
\\
  &
   + \langle \hat{\rho}^{n+\frac{1}{2}} \zeta^{n+\frac12} ,  D_x ( \td\rho^{n+1} + \td \rho^n ) \rangle .
\end{aligned}
  \label{nonlinear est-5-4}
\end{equation}
Meanwhile, the point-wise ratio bound~\eqref{propr:eq17-2} reveals that
\begin{equation}
  \Big\langle \frac{\hat{\rho}^{n+\frac12}}{2 \rho^{n+1}}  D_x ( \td \rho^{n+1} + \td \rho^n )  ,
   D_x ( \td\rho^{n+1} + \td \rho^n )  \Big\rangle
   \ge \frac38 \| D_x ( \td\rho^{n+1} + \td \rho^n ) \|_2^2 .
   \label{nonlinear est-5-5}
\end{equation}
In terms of the second term on the right hand side of~\eqref{nonlinear est-5-4}, a direct application of the Cauchy inequality indicates that
\begin{equation}
   \langle \hat{\rho}^{n+\frac{1}{2}} \zeta^{n+\frac12} ,  D_x ( \td\rho^{n+1} + \td \rho^n ) \rangle
   \ge - \frac18 \| D_x ( \td\rho^{n+1} + \td \rho^n ) \|_2^2 - 2 \|\hat{\rho}^{n+\frac{1}{2}} \zeta^{n+\frac12} \|_2^2 .
   \label{nonlinear est-5-6}
\end{equation}
A further application of the Cauchy inequality gives
\begin{equation}
\begin{aligned}
  \| \hat{\rho}^{n+\frac{1}{2}} \zeta^{n+\frac12} \|_2^2 \le &
  7 \tilde{C}_0^2 \breve{C}_1^2 (  \| \td \rho^{n+1} \|_2^2  + \| \td \rho^n \|_2^2 )
   + 7 \tilde{C}_0^2( \breve{C}_2^2 h^2  + \breve{C}_3^2 \dt^2 ) ( \|  D_x \td \rho^{n+1} \|_2^2
   + \| D_x \td \rho^n \|_2^2 )
\\
  &
   + \frac{28 (\tilde{C}_0 C^*)^2 ( \td C_0+C^* )^8 (\ve^*)^{-10} }{25} h^8
\\
  \le &
   7 ( \tilde{C}_0^2\breve{C}_1^2 + \tilde{C}_0^2\breve{C}_2^2 \hat{C}_4^2
   +  \tilde{C}_0^2\breve{C}_2^2 \hat{C}_3^2 ) (  \| \td \rho^{n+1} \|_2^2  + \| \td \rho^n \|_2^2 ) 
\\
  & 
   + \frac{28 (\tilde{C}_0 C^*)^2 ( \td C_0 +C^*)^8 (\ve^*)^{-10} }{25} h^8,
\end{aligned}
   \label{nonlinear est-5-7}
\end{equation}
in which the inverse inequalities, $\dt \| \nabla_h f \|_2 \le \hat{C}_3 \| f\|_2$ and $h \| \nabla_h f \|_2 \le \hat{C}_4 \| f\|_2$ (with the linear refinement requirement $\lambda_1 h \le \dt \le\lambda_2 h$), have been applied. Therefore, a substitution of \eqref{nonlinear est-5-5}-\eqref{nonlinear est-5-7} into \eqref{nonlinear est-5-4} yields
\[
\begin{aligned}
  \gamma \langle \hat{\rho}^{n+\frac{1}{2}} D_x \td S^{n+\frac12} , D_x ( \td\rho^{n+1} + \td \rho^n ) \rangle
  \ge   \frac{\gamma }{4} \| D_x ( \td\rho^{n+1} + \td \rho^n ) \|_2^2
  - \breve{C}_5  (  \| \td \rho^{n+1} \|_2^2  + \| \td \rho^n \|_2^2 )
   -  Q^{(0)} h^8,
\end{aligned}
\label{nonlinear est-5-8}
\]
with
\[
\begin{aligned}
  \breve{C}_5 = 14\gamma  ( \tilde{C}_0^2\breve{C}_1^2 + \tilde{C}_0^2\breve{C}_2^2 \hat{C}_4^2
   +  \tilde{C}_0^2\breve{C}_2^2 \hat{C}_3^2 ) ,  \quad
   Q^{(0)} = \frac{56\gamma  (\tilde{C}_0 C^*)^2 ( \td C_0 +C^*)^8 (\ve^*)^{-10} }{25}.  
\end{aligned}
\]
The nonlinear diffusion error estimates in the $y$ and $z$ directions could be performed in the same manner. This completes the proof of Proposition~\ref{prop:CA2}, by taking $\td C_1 = 3 \breve{C}_5$ and $M_1 = 3 Q^{(0)}$.

\vskip 5mm
\noindent{\bf Acknowledgements.}
This work is supported in part by the National Natural Science Foundation of China 12101264, Natural Science Foundation of Jiangsu Province BK20210443, High level personnel project of Jiangsu Province 1142024031211190 (J. Ding), National Science Foundation DMS-2012269 and DMS-2309548 (C. Wang), and National Natural Science Foundation of China 12171319 (S. Zhou).

\bibliographystyle{plain}
\bibliography{KS}

\end{document}